\documentclass[11pt]{article}
\usepackage{amssymb}
\usepackage{mathrsfs,pstricks}
\usepackage{amsfonts}
\usepackage{latexsym}
\usepackage{amsmath,amssymb}
\usepackage{amsthm}
\usepackage{amsfonts}
\usepackage{ifthen,calc}
\usepackage{color}
\usepackage{enumerate}
\usepackage{tikz}
\usepackage{epsfig}
\usepackage{graphicx}
\usepackage{float}
\usepackage{latexsym}
\usepackage{multicol}
\usepackage{subfigure}
\usepackage{geometry}
\usepackage{setspace}
\usepackage{CJK}
\usepackage{mathrsfs}
\usepackage{amsmath}
\usepackage{amssymb,mathrsfs,amsthm,amsfonts}
\usepackage{graphicx}
\usepackage[pdfstartview=FitH]{hyperref}
\usepackage[title]{appendix}
\usepackage{color}
\usepackage{tikz}
\usepackage{epsfig}
\usepackage{subfigure}

\providecommand{\U}[1]{\protect\rule{.1in}{.1in}}
\allowdisplaybreaks % ����ʵ�ֶ��й�ʽ�Զ���ҳ

\newtheorem{thm}{Theorem}[section]
\newtheorem{proposition}{Proposition}[section]
\newtheorem{corollary}{Corollary}[section]
\newtheorem{lemma}{Lemma}[section]
\newtheorem{definition}{Definition}[section]
\newtheorem{hypothesis}{Hypothesis}[section]

\def\mE{{\mathbb E}}
\def\eps{\varepsilon}
\def\]{{\Big]}}
\def\[{{\Big[}}
\def\dif{{\mathord{{\rm d}}}}
\def\bd{\begin{definition}}
\def\ed{\end{definition}}
\def\bp{\begin{proposition}}
\def\ep{\end{proposition}}
\def\bc{\begin{corollary}}
\def\ec{\end{corollary}}
\def\bx{\begin{Examples}}
\def\ex{\end{Examples}}
\def\cA{{\mathcal A}}
\def\cB{{\mathcal B}}
\def\cC{{\mathcal C}}
\def\cD{{\mathcal D}}
\def\cE{{\mathcal E}}
\def\cF{{\mathcal F}}

\def\cJ{{\mathcal J}}
\def\cK{{\mathcal K}}
\def\cL{{\mathcal L}}
\def\cM{{\mathcal M}}
\def\cN{{\mathcal N}}

\def\cQ{{\mathcal Q}}
\def\cR{{\mathcal R}}
\def\cS{{\mathcal S}}

\def\cX{{\mathcal X}}

\def\cZ{{\mathcal Z}}

\def\mE{{\mathbb E}}

\def\mI{{\mathbb I}}

\def\mN{{\mathbb N}}

\def\mP{{\mathbb P}}

\def\mR{{\mathbb R}}

\def\mT{{\mathbb T}}

\def\mZ{{\mathbb Z}}

\def\ba{{\begin{align}}
\def\ea{\end{align}}}

\def\geq{\geqslant}
\def\leq{\leqslant}

\def\^{\widehat}

\def\ba{\begin{aligned}}
\def\ea{\end{aligned}}
\def\be{\begin{equation}}
\def\ee{\end{equation}}
\def\ben{\begin{align*}}
\def\enn{\end{align*}}

\makeatletter

\newcommand{\Rmnum}[1]{\expandafter\@slowromancap\romannumeral #1@}
\makeatother

\numberwithin{equation}{section}

\title{\bf{Ergodicity and exponential mixing of the real Ginzburg-Landau equation with a degenerate noise$^*$}}

{\author{ Xuhui Peng$^\dag$, Jianhua Huang$^\spadesuit$, Rangrang Zhang$^\clubsuit$\\
{\em\small $^\dag$MOE-LCSM, School of mathematics and statistics,  Hunan Normal University}\\
{\em\small Changsha  410081, P.R.China} \\
{\em\small $^\spadesuit$College of Science, National  University of Defense Technology  }\\
{\em\small  Changsha, {\rm 410073}, P.R.China}\\
{\em\small $^\clubsuit$School of  Mathematics and Statistics,
Beijing Institute of Technology, Beijing, 100081, China. }\\
 }
 }

\date{}

\begin{document}

\maketitle

\let\thefootnote\relax\footnotetext{
$^*$The first author was supported by  Hunan Provincial Natural Science Foundation of
China (No.2019JJ50377), NSFC (No.11871476) 
 and  the Construct Program of the
Key Discipline in Hunan Province,
the second author was supported by NSFC (No.11371367, 11771449),the third author was supported by NSFC (No. 11801032), Key Laboratory of Random Complex Structures and Data Science, Academy of Mathematics and Systems Science, Chinese Academy of Sciences (No. 2008DP173182), China Postdoctoral Science Foundation funded project (No. 2018M641204).
}
\let\thefootnote\relax\footnotetext{$^\clubsuit$Corresponding author.}
\let\thefootnote\relax\footnotetext{Email addresses:xhpeng@hunnu.edu.cn(X.Peng),jhhuang32@nudt.edu.cn(J.Huang),rrzhang@amss.ac.cn(R.Zhang)}
%
% \let\thefootnote\relax\footnotetext{
%  Email:jinlian916@hnu.edu.cn(J.Zhang),xhpeng@hunnu.edu.cn(X.Peng),yichen@hnu.edu.cn(Y.Chen)
% }

%\footnote{$^*$ The first author
%was  supported by NSFC (No.11501195), a Scientific Research Fund of Hunan Provincial Education Department (No.17C0953), and the second author was  supported by NSFC(No.11371367, 11771449).}
%\footnote{ Email addresses:  xhpeng@hunnu.edu.cn(X.Peng), jhhuang32@nudt.edu.cn(J.Huang).}

\begin{abstract}
In this paper, we establish the existence, uniqueness and attraction properties of an invariant measure for the real Ginzburg-Landau equation  in the presence of a degenerate stochastic forcing acting only in four directions. The main  challenge is to establish time asymptotic smoothing properties of the Markovian dynamics corresponding to this system. To achieve this,  we propose  a condition which only requires  four     noises.
% Our method in this article is Malliavin calculus.

\vspace{1mm}
\end{abstract}
\section{Introduction and     Main Results }

\subsection{Introduction }
In this paper, we  are concerned with   the ergodicity of the stochastic real Ginzburg-Landau equation driven by Brownian motion on torus $\mT=\mR/{2\pi \mZ}$ as follows
\begin{eqnarray}\label{p-1}
\left\{
\begin{split}
  & \dif U-\frac{\partial^2U}{\partial z^2}\dif t-(U-U^3)\dif t=\sum_{k\in \cZ_0}\beta_ke_k\dif W_k(t),
  \\ &  U|_{t=0}=U_0,
  \end{split}
  \right.
\end{eqnarray}
where $U:[0,\infty)\times \mT\rightarrow \mR,$ $\cZ_0$ is a subset of $\mZ_*=\mZ\setminus \{0\},$
$\{\beta_k\}_{k\in \cZ_0}$ are non-zero constants,  $\{W_k(t)\}_{k\in\mZ}$ is one  dimensional real-valued i.i.d Brownian motion sequence defined on a filtered probability space $(\Omega,\cF,\{\cF_t\},\mP)$ and
  \begin{eqnarray*}
   e_k(z)=
   \left\{
   \begin{split}
     & \sin(kz), \quad k\in \mZ\cap [1,\infty), z\in\mT,
     \\
     & \cos(kz), \quad k\in \mZ\cap (-\infty,-1], z\in \mT.
   \end{split}
   \right.
 \end{eqnarray*}

Consider the following abstract equation on a Hilbert space $H,$
\begin{eqnarray*}
  \dif U=F(U)\dif t+G\dif W_t,\quad U|_{t=0}=U_0.
\end{eqnarray*}
There is a wide literature devoted to proving  uniqueness and associated mixing properties of invariant measures for nonlinear stochastic PDEs when $GG^*$ is non-degenerate or  mildly degenerate (see e.g. \cite{E-2000,KS,MY-2002,PZ-92,PZ-96} and  references therein).

The purpose of this paper is to  prove the exponential mixing for   stochastic real Ginzburg-Landau equation (\ref{p-1})   when  the random  forcing is extremely degenerate to be several noises.
 There are several  works  related to  this topic    when  the random  forcing is extremely degenerate.  %noise.
 We  mention some of them   which  are relevant	 to our work.

 \begin{itemize}
   \item Hairer and  Mattingly \cite{martin,Hairer02}  considered   stochastic 2D Navier-Stokes equations on a torus driven by degenerate additive noise.
They   established   an exponential mixing property  of the solution  of
the  vorticity formulation for the  2D stochastic  Navier-Stokes equations  by using Malliavin calculus, although the noise  is extremely degenerate ( the noise only acts in four directions).
   \item
  F\"oldes   et al. \cite{FGRT} was interested in  the  following  stochastic   Boussinesq equations
  \begin{eqnarray}\label{p-41}
  \left\{
  \begin{split}
   &  \dif u+(u \cdot \nabla u)\dif t=(-\nabla p+\nu_1 \Delta u+\textbf{g}\theta )\dif t, ~~~\nabla \cdot u=0
    \\  & \dif \theta+(u\cdot \nabla \theta)\dif t=\nu_2 \Delta \theta \dif t +\sigma_\theta \dif W,
    \end{split}
    \right.
  \end{eqnarray}
 where $u=(u_1,u_2)$ denotes the velocity field, $\theta$  is the   temperature, $\textbf{g}=(0,g)^T$ with $g\neq 0$ is a constant. The authors worked on
 the  vorticity equations of  (\ref{p-41}), which is given by
 \begin{eqnarray}\label{p-44}
\left\{
\begin{split}
 &  \dif \omega+(u\cdot \nabla \omega-\nu_1\Delta \omega)=g\partial_x\theta\dif t,
 \\
 &
 \dif \theta+(u\cdot \nabla \theta-\nu_2\Delta \theta)=\sigma_\theta \dif W.
  \end{split}
  \right.
\end{eqnarray}
Although the forcing is extremely degenerate(only four directions in $\theta$ have noise), the authors succeed to  establish  an exponential mixing property for the solution  of
 equation (\ref{p-44}) by utilizing Malliavin calculus.
 \end{itemize}

As stated above, all the authors in  \cite{FGRT,martin} established an exponential mixing property for the solution  of vorticity equation instead of velocity equation. For our model, we can directly deal with the velocity equation (\ref{p-1}) due to its special structure.
Let   $U_t$ be the solution to equations  (\ref{p-41}) or (\ref{p-44}) and  $\cJ_{0,t}\xi=DU_t(x)\xi$ be  the effect on
$U_t$ of an infinitesimal perturbation of the initial condition in
the direction $\xi$.   The authors of \cite{FGRT,martin} considered  the vorticity formulation in order to obtain $\mE \|\cJ_{0,t}\xi\|^p<\infty$.  For equation  (\ref{p-1}), we can directly achieve it .

For the stochastic real Ginzburg-Landau equation, we mention   the following results.
 \begin{itemize}
   \item For  the stochastic real Ginzburg-Landau equation driven by Brownian motion,
   Hairer \cite[Section 6]{Hairer-2001}
    established  an exponential mixing of the solution  to (\ref{p-1}) under the condition that the number of   noises can be finite but should be sufficiently many. Our results in this article are stronger than that. Meanwhile,  the random  forcing of our model can  be  extremely degenerate to be only several  noises.
   \item  Xu \cite{Xu-2012} proved that  the stochastic real Ginzburg-Landau equation driven by $\alpha$-stable  process   admits a  unique invariant measure under some conditions.  The noise in \cite{Xu-2012} is required to be non-degenerate.
   \item  Mourrat and Weber \cite{MW17} established   a priori estimates for  the dynamic $\Phi_3^4$ model on the torus which is independent of initial conditions. The $\Phi_3^4$ model
     is formally given by the stochastic partial differential equation
     \begin{eqnarray*}
     \left\{
       \begin{split}
         &\partial_t X=\Delta X-X^3+mX+\xi, \quad \text{on }\mR_+\times [-1,1]^3,
         \\ &X(0,\cdot)=X_0
       \end{split}
       \right.
     \end{eqnarray*}
     where $\xi$ denotes a white noise on $\mR_+\times [-1,1]^3$, and $m\in \mR$ is a parameter.
  \end{itemize}

 \subsection{Main results}

Let $\mT=\mR/{2\pi \mZ}$ be equipped with the usual Riemannian metric, and let $\dif z$ denote the Lebesgue measure on $\mT$. Then
\begin{eqnarray*}
  H:=\left\{\xi\in L^2(\mT,\mR); \int_{\mT}\xi(z)\dif z=0\right\}
\end{eqnarray*}
is a separable real Hilbert space with inner product
 \begin{eqnarray*}
 \langle \xi,\eta \rangle =\int_{\mT}\xi(z)\eta(z)\dif z,\quad \quad \forall \xi,\eta\in H
 \end{eqnarray*}
 and norm $\|\xi\| =\langle \xi,\xi \rangle^{1/2}. $

  It is well-known that
 \begin{eqnarray*}
   \{e_k:k\in\mZ_*\}
 \end{eqnarray*}
 is an orthonormal basis of $H$.
 For each $x\in H$, it can be represented by
 \begin{eqnarray*}
   x=\sum_{k\in \mZ_*}x_ke_k.
 \end{eqnarray*}

  Let $\Delta=\frac{\partial^2}{\partial z^2}$ be the Laplace operator on $H$, then
 \begin{eqnarray}
 \label{zhang-1}
   \Delta e_k=-\gamma_k e_k,  \text{ with }  k\in\mZ_*,\gamma_k=|k|^2.
 \end{eqnarray}
 For $\sigma>0,$ we define
 \begin{eqnarray*}
   A&:=&-\Delta,
   \\ H^\sigma= H^{\sigma,2}(\mT)&:=&\left\{x\in H:~x=\sum_{k\in\mZ_*}x_ke_k\text{ with } \|x\|_{H^\sigma}^2:=\sum_{k\in Z_*}(1+ \gamma_k )^{\sigma }|x_k|^2<\infty\right\},
   \\ D(A^{\sigma})&:=&\left\{x\in H:~x=\sum_{k\in\mZ_*}x_ke_k \text{ with } \|x\|^2_{D(A^\sigma)}:=\sum_{k\in Z_*}|\gamma_k|^{2\sigma}|x_k|^2<\infty \right\},
   \\ V^\sigma&:=&\left\{x\in D(A^{\sigma/2}),\text{ with } \|x\|_\sigma:=\|x\|_{D(A^{\sigma/2})}<\infty\right\}.
 \end{eqnarray*}
For $\sigma>0,$ we denote by  $H^{-\sigma}$ the dual space of $H^{\sigma}.$
For the sake of convenience,  we  denote by  $V=V^1$.

Set $
  N(U)=-U+U^3
$
and
\begin{eqnarray*}
  F(U)=-A U-N(U)=\Delta U+U-U^3.
\end{eqnarray*}
Let $\{\theta_k\}_{k\in \cZ_0}$ be the standard basis of $\mR^{|\cZ_0|}$,   where $|\cZ_0|$ denotes the number of the element belongs  to the set $\cZ_0.$
We define a linear map $G:\mR^{|\cZ_0|}\rightarrow H$ such that
\begin{eqnarray}\label{equu-1}
  G\theta_k=\beta_ke_k,
\end{eqnarray}
where  $\{\beta_k\}_{k\in\cZ_0}$ is  a sequence of non-zero numbers appeared in (\ref{p-1}).
We consider the stochastic forcing of the form
\begin{eqnarray*}
G \dif W_t=\sum_{k\in  \cZ_0 }\beta_{k}e_k \dif W_{k}(t),
\end{eqnarray*}
then  (\ref{p-1}) can be written as
\begin{eqnarray*}
  \dif U=F(U)\dif t+G
\dif W, \quad U|_{t=0}=U_0.
\end{eqnarray*}
For any $n\geq 1,$ we define $\cZ_n$ recursively as follows:
\begin{eqnarray}\label{pp-10}
  \cZ_n:=\{k+\ell+m:~ k\in \cZ_{n-1},\ell,m\in\cZ_0\}.
\end{eqnarray}

Our Hypothesis in this article is
\begin{hypothesis}
\label{pp-2}
\

\begin{itemize}
  \item[({\romannumeral1}) ]
  if $k\in\cZ_0$, then $-k\in \cZ_0$,
  \item[({\romannumeral2}) ]
  $\cup_{n=0}^\infty\cZ_n=\mZ_*.$
  \item[({\romannumeral3}) ]
 $ |\cZ_0|<\infty.$
\end{itemize}
\end{hypothesis}

To measure the convergence to equilibrium, we will use the following distance
 function on $H$
 \begin{eqnarray}
 \label{5-1}
   d(x,y)=1\wedge \delta^{-1}\|x-y\|.
 \end{eqnarray}
 where $\delta$ is a small parameter to be adjusted later on. The distance (\ref{5-1}) extends in a
natural way to a Wasserstein distance between probability measures by
\begin{eqnarray*}
  d(\mu_1,\mu_2)=\sup_{\|\Phi\|_{d}\leq 1}\left|\int_{H}\Phi(x)\mu(\dif x)-\int_{H}\Phi(x)\nu(\dif x)\right|
\end{eqnarray*}
 where $\|\Phi\|_d$ denotes the Lipschitz constant of $\Phi$ in the metric $d$.

The transition function associated to (\ref{p-1}) is given by
\begin{eqnarray}
\label{zhang-1}
  P_t(U_0,E)=\mP(U(t,U_0)\in E) \text{ for any } U_0 \in H, E\in \cB(H), t\geq 0,
\end{eqnarray}
where $\cB(H)$ is the collection of Borel sets on $H$, $U(t,U_0)$ is the solution to equations (\ref{p-1}) with initial value $U_0\in H$.
We also define the Markov semigroup $\{P_t\}_{t\geq 0}$ with $P_t:M_b(H)\rightarrow M_b(H)$ associated to (\ref{p-1}) by
\begin{eqnarray}\label{p-26}
  P_t \Phi (U_0) :=\mE \Phi(U(t,U_0))=\int_H \Phi(\bar{U})P_t(U_0,\dif \bar{U}) \text{ for any } \Phi \in M_b(H), t\geq 0,
\end{eqnarray}
where  $M_b(H)$ is  the space of bounded measurable  functions on $H$ equipped with supremum norm.
Denote by  $C_b(H)$  the space of   bounded continuous real-valued functions on $H$.
Let $Pr(H)$  be the collection of Borelian probability measures on $H.$ The dual operator $P_t^*$ of $P_t$, which maps $Pr(H)$ to itself, is given by
\begin{eqnarray}
\label{e-1}
  P_t^*\mu(A):=\int_H P_{t}(U_0,A)\dif \mu(U_0),
\end{eqnarray}
over $\mu\in Pr(H).$

Now we will give our main results in this paper.
\begin{thm}\label{p-27}
Assume Hypothesis  \ref{pp-2} holds, then there exists a unique invariant measure $\mu_*$ associated to (\ref{p-1}) and for each $t\geq 0$ the map $P_t$ is ergodic related to  $\mu_*$.
Concretely, the following results
hold.
\begin{itemize}
  \item[({\romannumeral1}) ] (Exponential Mixing) There are constants     $\delta>0$ and  $\gamma>0$  such that
  \begin{eqnarray}\label{p-31}
    \sup_{\|\Phi\|_d\leq 1}\left| \mE \Phi(U(t,U_0))-\int_H \Phi(\bar{U})\dif \mu_*(\bar{U}) \right|
    \leq  Ce^{-\gamma t},
  \end{eqnarray}
  where $C$ is a constant independent of $U_0$ and $t.$
   \item[({\romannumeral2}) ] (Weak law of large numbers)
   For the  $\delta>0$ in ({\romannumeral1}),  any $\Phi$ with $ \|\Phi\|_d\leq 1$ and any $U_0\in H$, we  have
   \begin{eqnarray}\label{p-35}
     \lim_{T\rightarrow \infty}\frac{1}{T}\int_0^T \Phi(U(t,U_0))\dif t=\int_H \Phi(\bar{U})\dif \mu_*(\bar{U})=:m_{\Phi} \quad \text{  in probability.}
   \end{eqnarray}
  \item[({\romannumeral3}) ] (Central limit theorem) For the $\delta>0$ in ({\romannumeral1}),  any $\Phi$ with $ \|\Phi\|_d\leq 1$,  every $U_0\in H$ and  $\xi\in \mR$, we have
      \begin{eqnarray}\label{p-36}
        \lim_{T\rightarrow \infty}\mP\left(\frac{1}{\sqrt{T}}\int_0^T (\Phi(U(t,U_0))-m_{\Phi})\dif t<\xi\right)=\cX(\xi),
      \end{eqnarray}
      where $\cX$ is the distribution function of a normal random variable whose mean is equal to zero and variance is equal to
      \begin{eqnarray*}
          \lim_{T\rightarrow \infty}\frac{1}{T}\mE \left(\int_0^T (\Phi(U(t,U_0))-m_{\Phi}) \dif t\right)^2.
      \end{eqnarray*}
\end{itemize}
\end{thm}
We emphasis that  the constant   $C$ appeared   in (\ref{p-31}) is independent of  the initial value  $U_0.$
This is  one of the  challenges in our  paper.

Based on Theorem \ref{p-27}, the following result  holds.
\begin{corollary}
For any $n\geq 1,$  if  $ \cZ_0=\{ -(n+1),-n,n,n+1\}$, the  results of  Theorem  \ref{p-27} hold.
\end{corollary}

\subsection{The organization of this paper}
This article is organized as follows:
Section 2 is devoted to establishing some moment estimates.
In Section 3,  we  present the  proof of  spectral properties  for the Malliavin matrix $\cM_{0,t}$ of $U_t$  in  Theorem \ref{p-9}   and   demonstrate   a
gradient estimate of   $P_t$  in Proposition \ref{p-28}.
Finally, we give a proof of Theorem  \ref{p-27} in Section 4.

\section{Some moment estimates}

In this section, we  establish  some  moment estimates which are useful in this paper. When $T>0$ is a constant,  we always denote by  $C_T$  a constant depending on $T$ and it may changes from line to line.

We say that $U_t=U(t,U_0)$ is a solution to (\ref{p-1}) if it is $\cF_t$-adapted,
\begin{eqnarray}\label{pp-16}
  U\in C([0,\infty),H)\cap L_{loc}^2([0,\infty),V)\quad a.s.,
\end{eqnarray}
and $U$ satisfies (\ref{p-1}) in the mild sense, that is
\begin{eqnarray*}
  U_t=e^{-At}U_0-\int_0^te^{-A(t-s)}N(U_s)\dif s+\int_0^t e^{-A(t-s)}G\dif W_s.
\end{eqnarray*}

The following proposition summarizes the basic well-posedness, regularity, and smoothness of equation (\ref{p-1}).
\begin{proposition}
\label{9-1}
Given any $U_0\in H$, there exists a unique solution $U:[0,\infty)\times \Omega\rightarrow H$ of (\ref{p-1}) which is an $\cF_t$-adapted process on $H$ satisfying (\ref{pp-16}).

For any $t\geq 0$ and any realization of the noise $W(,\omega),$ the map $U_0\mapsto U(t,U_0)$ is Fr\'echet differential on $H$. For every fixed $U_0\in H$ and $t\geq 0,$
$W\mapsto U(t,W)$ is Frechet differential from $C((0,t), \mR^{|\cZ_0|})$ to $H$. Moreover, $U$ is spatially smooth for all positive time, that is, for any $t_0>0$ and any $s>0,$
\begin{eqnarray*}
  U\in C([t_0,\infty),H^s)\quad a.s..
\end{eqnarray*}
\end{proposition}

Since we are considering the case of spatially smooth, additive noise, the proof of the well-posedness  of (\ref{p-1}) is standard and can be obtained following along the line of classical proof for the stochastic 2D Navier-Stokes equations (see e.g. \cite{KS2012}).

Let $U_t=U(t, U_0,W)$ be the solution of  (\ref{p-1}) with initial value $U_0$ and noise $W$. For any $\xi\in H$  and $s\geq 0,$ $\cJ_{s,t}\xi$ denotes the unique solution of
\begin{eqnarray}\label{p-37}
\left\{
\begin{split}
  & \partial_t  \cJ_{s,t}\xi+ A\cJ_{s,t}\xi -\cJ_{s,t}\xi+3U_t^2\cJ_{s,t}\xi=0,
  \\ & \cJ_{s,s}\xi=\xi.
  \end{split}
  \right.
\end{eqnarray}
The Malliavin derivative $\cD:L^2(\Omega;H)\rightarrow L^2(\Omega, L^2(0,T,\mR^{|\cZ_0|})\times H)$ satisfies that   for each $v\in L^2(0,T,\mR^{|\cZ_0|}) $
\begin{eqnarray*}
  \langle \cD U,v\rangle_{ L^2(0,T,\mR^{|\cZ_0|}) }=\lim_{\eps \rightarrow 0}\frac{1}{\eps}
  \Big(U(T,U_0,W+\eps \int_0^\cdot v\dif s)-U(T,U_0,W)\Big),
\end{eqnarray*}
 we may infer that for $v\in L^2(\Omega, L^2(0,T,\mR^{|\cZ_0|})) ,$
 \begin{eqnarray*}
  \langle \cD U,v\rangle_{ L^2(0,T;\mR^{|\cZ_0|}) }=\int_0^T\cJ_{s,T}G v(s)\dif s,
\end{eqnarray*}
and hence, by the Riesz representation theorem,
\begin{eqnarray*}
  \cD^j_sU_T=\cJ_{s,T}G\theta_j, \text{ for any } s\leq T,j=1,\cdots, |\cZ_0|.
\end{eqnarray*}
Here and below, we adopt the standard notation $\cD^j_sF :=(DF)^j(s)$, that is,
 $\cD^j_sF$ is the $j$-th component of $\cD F$ evaluated at time s.

We define the random operator $\cA_{s,t}:L^2(s,t,\mR^{|\cZ_0|})\rightarrow H$ by
\begin{eqnarray*}
\cA_{s,t}v:=\int_{s}^{t}\cJ_{r,t}Gv(r)\dif r.
\end{eqnarray*}
Notice that, for any $0\leq s<t$, the function $\varrho(t):=\cA_{s,t}v$ satisfies the following equation
\begin{eqnarray*}
  \left\{
\begin{split}
  & \partial_t  \varrho(t)  + A \varrho(t) -\varrho(t)+3U_t^2 \varrho(t)=Gv(t),
  \\ & \varrho(s)=0.
  \end{split}
  \right.
\end{eqnarray*}
For any $s<t$, let $\cA_{s,t}^*:H\rightarrow L^2(s,t,\mR^{|\cZ_0|})$  be the adjoint of $\cA_{s,t}$, then
\begin{eqnarray*}
(\cA_{s,t}^*\xi)(r)=G^{*}K_{r,t}\xi,  \text{ for any }\xi\in H, r\in [s,t]
\end{eqnarray*}
where $G^*:H\rightarrow \mR^{|\cZ_0|}$ is the adjoint of $G$, and for $s<t,\cK_{s,t}\xi=\cJ_{s,t}^*\xi$ is the solution of the following ``backward"  system
\begin{eqnarray}\label{p-3}
 \partial_s \varrho^* =A\varrho^*+(\nabla N(U_s))^{*}\varrho^*=-(\nabla F(U_s))^*\varrho^*,\quad \varrho^*(t)=\xi.
\end{eqnarray}

We then define the Malliavin matrix
\begin{eqnarray}
\label{10-1}
  \cM_{s,t}:=\cA_{s,t}\cA_{s,t}^*:H\rightarrow H.
\end{eqnarray}
Observe that $\rho_t:=\cJ_{0,t}\xi-\cA_{0,t}v$ satisfies
\begin{eqnarray}
\label{a-1}
\left\{
\begin{split}
  & \partial_t  \rho_t+ A \rho_t-\rho_t+ 3U_t^2 \rho_t=-Gv(t),
  \\ & \rho(0)=\xi.
  \end{split}
  \right.
\end{eqnarray}

For any $t\geq s\geq 0$ let $\cJ_{s,t}^{(2)}:H\rightarrow \cL(H,\cL(H))$
be the second derivative of $U$ with respect to an initial value  $U_0$.
In this paper,  $\cL(X)=\cL(X,X)$ and  $\cL(X,Y)$  is the space of linear operators from   $X$ to $Y.$
 Observe
that for fixed $U_0\in H$ and any $\xi,\xi'\in H$ the function $\varrho_t:=\cJ_{s,t}^{(2)}(\xi,\xi')$ is the solution of
\begin{eqnarray*}
  \partial_t \varrho_t +A \varrho_t  -\varrho_t+3U_t^2\varrho_t+6U_t\cJ_{s,t}\xi \cJ_{s,t}\xi'=0,~\quad \varrho(s)=0.
\end{eqnarray*}

For any $\alpha \in (0,1]$ and function $g:[T/2,T]\rightarrow \mR$, ~$\|g\|_{C^\alpha[T/2,T]}$ is defined by
\begin{eqnarray*}
\|g\|_{C^\alpha[T/2,T]}:=\sup_{\mbox{\tiny$\begin{array}{c}
t_1\neq t_2\\
t_1,t_2\in [T/2,T]
\end{array}$}}
\frac{|g(t_1)-g(t_2)|}{|t_1-t_2|^{\alpha}}.
\end{eqnarray*}
For any $\alpha \in (0,1]$ and function $f:[T/2,T]\rightarrow H$,  we  define the semi-norms
\begin{eqnarray*}
  \|f\|_{C^\alpha([T/2,T],H)}:=
  \sup_{\mbox{\tiny$\begin{array}{c}
t_1\neq t_2\\
t_1,t_2\in [T/2,T]
\end{array}$}}
\frac{\|f(t_1)-f(t_2)\|}{|t_1-t_2|^{\alpha}}.
\end{eqnarray*}
%For $\alpha=0,$ $\|g\|_{C^\alpha}$ is defined by
%\begin{eqnarray*}
%\|g\|_{C^\alpha}:=\|g\|_{C^\alpha[T/2,T]}:=\sup_{t\in [T/2,T]}|g(t)|.
%\end{eqnarray*}
%Similar notations will be employed for  $\|U\|_{C^{1,\alpha}([T/2,T],H^\beta)}$ etc.

\begin{lemma}
\label{4-1}
  For any $m>0,T>0,$ there exists a positive constant $\gamma=\gamma_{m,T} $ such that
  \begin{eqnarray}
  \label{3-1}
    \mE \| U_t^{m}\|^2\leq C_{T,m}(t^{-\gamma}+1), \quad \forall t\in (0,T].
  \end{eqnarray}
  and
     \begin{eqnarray}
     \label{3-2}
    \mE \| U_t^{m}\|^2 \leq C_{T,m}( \|U_0^{m}\|^2+1), \quad \forall t\in (0,T],
  \end{eqnarray}
  where $C_{T,m}$ is a constant depending on $T$ and $m$.
\end{lemma}
\begin{proof}
Applying $It\^{o}$ formula to  $f(t)=\langle U_t,U_t^{2m-1}\rangle$, it gives
\begin{eqnarray*}
  \dif \|U_t^m\|^2&\leq&  \langle \dif U_t,U_t^{2m-1}\rangle+\langle  U_t,(2m-1) U_t^{2m-2}\dif U_t \rangle+C_m \|U_t^{m-1}\|^2\dif t
  \\ &\leq &  \langle \dif U_t,2m U_t^{2m-1}\rangle+C_m (1+\|U_t^{m}\|^2)\dif t
    \\ &=&  \langle (\Delta U_t+U_t-U_t^3)\dif t ,2m U_t^{2m-1}\rangle+C_m (1+ \|U_t^{m}\|^2)\dif t+\dif  M_t
    \\ &=& -(2m)(2m-1)\|\partial_zU \cdot U_t^{m-1}\|^2+2m \|U_t^m\|^2-2m \|U_t^{m+1}\|^2+C_m(1+ \|U_t^{m}\|^2)\dif t+\dif M_t,
\end{eqnarray*}
where $M_t$ is a martingale, $C_m$ is a constant depending on $m$ and $\beta_k$.
For any $s\leq t\leq T,$
by Young's inequality $\|U_t^m\|^2\leq C_{\eps,m} +\eps \|U_{t}^{m+1}\|^2, \forall \eps>0$, one arrives at
\begin{eqnarray}
\label{1-2}
  \mE \|U_t^m\|^2+m \mE \int_s^t\|U_r^{m+1}\|^2\dif r \leq C_{m,T}(1+\mE \|U_s^m\|^2).
\end{eqnarray}
Note that  $$(|a|+|b|)^{p}\leq 2^{p-1}(|a|^p+|b|^p), \ \forall p>1,$$
 we deduce that
\begin{eqnarray}
\label{1-1}
\int_s^t \big(\mE \| U_{r}^m\|^2+1 \big)^\lambda \dif r \leq C_{m,T} \big(\mE \| U_{s}^m\|^2+1 \big),
\end{eqnarray}
where $\lambda=\lambda(m)=\frac{m+1}{m}>1.$
By  \cite[Lemma 7.3]{MW17}, there exist an integer $N\geq 1$
and a sequence $0=t_0<t_1<t_2<\cdots<t_N=T$ such that for every $k\in\{0,1,\cdots,N-1\} $
\begin{eqnarray*}
\mE  \| U_{t_k}^m\|^2  \leq C_{T,m}\big(t_{k+1}^{-\frac{1}{\lambda-1}}+1).
\end{eqnarray*}
For any $t\in [t_k,t_{k+1}),$ by  (\ref{1-2}), we obtain
\begin{eqnarray*}
  \mE  \| U_{t}^m \|^{2} \leq C_{T,m}\big(\mE \| U_{t_k}^m\|^{2}+1 \big) \leq C_{T,m} \big(t_{k+1}^{-\frac{1}{\lambda-1}}+1) \leq C_{T,m}\big(t^{-\frac{1}{\lambda-1}}+1)
\end{eqnarray*}
which implies  (\ref{3-1}).

Let $s=0$ in (\ref{1-2}),  we obtain the desired result (\ref{3-2}).
\end{proof}

Define
$
  \cE(t)=\|U_t\|^2+ \int_0^t \|U_s\|_1^2\dif s
$
and
\begin{eqnarray}
\label{s-1}
 \mathfrak{B}_n=\sum_{k\in\cZ_0}\gamma_k^n \beta_k^2
 \end{eqnarray}
for $n\in \mN\cup\{0\},$ where  $\gamma_k$ is defined by (\ref{zhang-1}) and $\beta_k$ is in (\ref{equu-1}).
\begin{lemma}
\label{2-1}
  For any $m\geq 0,T\geq 0$, there exist some $C=C_{T,m}$ such that
  \begin{eqnarray*}
  \mE \sup_{t\in [0,T]}\cE (t)^m \leq C ( \|U_0\|^{2m}+1).
  \end{eqnarray*}
\end{lemma}
\begin{proof}
Let us set
$$
M_t=2\sum_{k\in\cZ_0} \beta_k\int_0^t \langle U_r,  e_k\rangle \dif W_k(r)
$$
By $It\^{o}$ formula, we have
\begin{eqnarray}
 \nonumber \|U_t\|^2&=& \|U_0\|^2 +\int_0^t 2\langle U_r, \dif U_r\rangle+\mathfrak{B}_0t
  \\ \nonumber &=& \|U_0\|^2+\int_0^t 2\langle U_r, -AU_r+U_r-U_r^3 \rangle\dif r + \mathfrak{B}_0t+M_t
    \\ \label{1-6} &=& \|U_0\|^2+\int_0^t \big(-2 \|U_r\|_1^2+2\|U_r\|^2-\|U_r\|_{L^4}^4 \big)\dif r + \mathfrak{B}_0t+M_t.
\end{eqnarray}
Note that the quadratic variation of $M_t$ is equal to
\begin{eqnarray*}
  \langle M\rangle_t=4\sum_{k\in \cZ_0} \beta_k^2\int_0^t \langle U_s,e_k\rangle^2\dif s
  \leq \gamma \int_0^t \| U_s\|^2 \dif s,
\end{eqnarray*}
where $\gamma=4\sum\limits_{k\in \cZ_0} \beta_k^2.$
We rewrite (\ref{1-6}) as follows
\begin{eqnarray*}
  \cE(t)- \mathfrak{B}_0t=\|U_0\|^2+M_t-\frac{1}{2}\gamma \langle M \rangle_t +K_t
\end{eqnarray*}
where
\begin{eqnarray*}
  K_t=\int_0^t \big(-\|U_r\|_1^2+2\|U_r\|^2-\|U_r\|_{L^4}^4 \big)\dif r+\frac{1}{2}\gamma \cdot \langle M \rangle_t\leq C_\gamma t.
\end{eqnarray*}
In the above, $C_\gamma$ is some constant depending on  $\gamma$ and in the last inequality,  we have used  $$\|U_r\|^2 \leq C_\eps +\eps  \|U_r\|_{L^4}^4, \quad \forall \eps>0.$$
Therefore,
\begin{eqnarray*}
 \cE(t)- (\mathfrak{B}_0+C_\gamma )t\leq  \|U_0\|^2+M_t-\frac{\gamma}{2}\langle M\rangle_t.
\end{eqnarray*}
By the supermartingale  inequality (cf. \cite[(7.57)]{KS2012}), we have
\begin{eqnarray*}
 &&  \mP\left(\sup_{t\in[0,T]}\big(\cE(t)- (\mathfrak{B}_0+C_\gamma t)\big)\geq \rho+\|U_0\|^2\right)
  \\ &&\leq \mP\left(\exp\left\{\gamma M_t-\frac{\gamma^2}{2}\langle M\rangle_t\right\}\geq e^{\gamma \rho}\right)
  \\&&\leq e^{-\gamma \rho}.
\end{eqnarray*}
Note that if $\xi$ and $\eta$ are non-negative random variables, then
\begin{eqnarray*}
  \mE \xi^m &\leq&    2^{m-1}\big(\mE (\xi-\eta)^m\mI_{\{\xi>\eta\}}+\mE \eta^m\big)
  \\ &=& 2^{m-1}\int_0^\infty \mP(\xi-\eta>\lambda^{1/m}) \dif \lambda+2^{m-1} \mE \eta^m.
\end{eqnarray*}
Apply this inequality  to $\xi=\sup_t\cE(t)$ and $\eta=\mathfrak{B}_0+C_\gamma t+\|U_0\|^2,$ we derive
\begin{eqnarray*}
  \mE \sup_{t\in[0,T]}\cE(t)^m\leq 2^{m-1}\int_0^\infty \exp{(-\gamma \lambda^{1/m})}\dif \lambda+2^{m-1}\mE (\mathfrak{B}_0+C_\gamma t+\|U_0\|^2)^m
\end{eqnarray*}
which yields the desired result.
\end{proof}

For any integer  $n\geq 0$, we set
\begin{eqnarray*}
  \cE(n,t)=t^n \|U_t\|_n+ \int_0^t s^n \|U_s\|_{n+1}^2 \dif s.
\end{eqnarray*}

\begin{lemma}
\label{1-8}
  For any $n, m\geq 0$, there exists  a constant $\kappa=\kappa_{n,m}$ such that
  \begin{eqnarray}
  \label{2-2}
  \mE \sup_{t\in [0,T]}\cE (n,t)^m \leq C_{n,m,T} ( \|U_0^\kappa\|^2+1).
  \end{eqnarray}
\end{lemma}
\begin{proof}
The proof is based on the method of induction in $n$.
Let us set $f_n(t)=t^n\langle A^n U_t,U_t\rangle$.
By the  $It\^{o}$ formula in \cite[Theorem 7.7.5]{KS2012} and  following similar arguments in the proof of \cite[Proposition 2.4.12]{KS2012}, we have
\begin{eqnarray}\label{3-3}
  f_n(t)=\int_0^t ns^{n-1}\|U_s\|_n^2+2s^n\langle A^nU_s,-AU_s+U_s-U_s^3\rangle+\mathfrak{B}_ns^n\dif s+M_t,
\end{eqnarray}
where
\begin{eqnarray*}
  M_t=\sum_{k\in \cZ_0} 2 \beta_k \int_0^t \langle A^nU_s,e_k\rangle \dif W_k(s)
  =\sum_{k\in \cZ_0} 2 \beta_k \gamma_k^n \int_0^t \langle U_s,e_k\rangle \dif W_k(s)
\end{eqnarray*}
The quadratic variation of $M_t$ is equal to
\begin{eqnarray}
\label{11-1}
  \langle M\rangle_t=4\sum_{k\in \cZ_0} \beta_k^2\gamma_k^{2n} \int_0^t \langle U_s,e_k\rangle^2\dif s
  \leq \gamma  \int_0^t \| U_s\|^2 \dif s,
\end{eqnarray}
where $\gamma=4\sum_{k\in \cZ_0} \beta_k^2\gamma_k^{2n}.$

Obviously, we have the following  identities
\begin{eqnarray}
  \label{3-4}\langle A^n U_s,AU_s\rangle= \|U_s\|_{n+1}^2,\quad
  \langle A^n U_s,U_s\rangle= \|U_s\|_n^2.
\end{eqnarray}
Firstly, we consider the case $n=1.$
In view of
\begin{eqnarray}
  \label{3-7}
  \begin{split}
  &\langle A U_s,U_s-U_s^3\rangle=\int_{\mT} -\frac{\partial^2  U_s(z)}{\partial z^2}(U_s(z)-U_s(z)^3)\dif z
  \\ & =\int_{\mT} \frac{\partial  U_s(z)}{\partial z}(\frac{\partial  U_s(z)}{\partial z}-3U_s(z)^2\frac{\partial  U_s(z)}{\partial z})\dif z\leq  \|U_s\|_1^2
  \end{split}
\end{eqnarray}
and by (\ref{3-3})(\ref{3-4}), we obtain
\begin{eqnarray*}
  t\|U_t\|_1^2+\int_0^ts\|U_s\|_2^2\dif s\leq C_T \int_0^t\|U_s\|_1^2\dif s+Ct^2+M_t.
\end{eqnarray*}
By (\ref{11-1}), we rewrite the above equality  in the form
\begin{eqnarray*}
  \cE(1,t)
    & \leq &   C_T \int_0^t \|U_s\|_1^2 \dif s +\frac{\gamma}{2}\langle M\rangle_t+C_T  t
  +M_t-\frac{\gamma}{2}\langle M\rangle_t
  \\ &\leq &     C_T \int_0^t \|U_s\|_1^2 \dif s +\frac{\gamma^2}{2}  \int_0^t \| U_s\|^2 \dif s+C_T  t
  +M_t-\frac{\gamma}{2}\langle M\rangle_t
    \\ &\leq &     C_T\cE(t)+C_T  t
  +M_t-\frac{\gamma}{2}\langle M\rangle_t.
\end{eqnarray*}
Combining the above  inequality with Lemma \ref{2-1} and following a similar argument as in the proof of   Lemma \ref{2-1}, we finish the proof of the inequality
(\ref{2-2}) with $n=1.$

Now, assume that for $k\leq n-1$, the inequality (\ref{2-2})  holds.
By Sobolev embedding theorem, we have
\begin{eqnarray}\notag
 \langle A^n U_s,U_s^3\rangle&=& \sum_{|\alpha|=n}C_\alpha \langle D^\alpha U^3, D^\alpha U\rangle
 \\   \label{3-6}
   &\leq& C(1+\|U_s\|_\infty^2)\|U_s\|_n^2
 \leq C(1+\|U_s\|_1^2)\|U_s\|_n^2,
\end{eqnarray}
where  $\|U_s\|_\infty=\sup_{z\in \mT}|U_s(z)|$.

By utilizng (\ref{3-6}) and (\ref{11-1})(\ref{3-4}),
we rewrite (\ref{3-3}) in the form
\begin{eqnarray*}
  \cE(n,t)&=& \int_0^t\Big[ ns^{n-1}\|U_s\|_n^2-s^n\|U_s\|_{n+1}^2+ 2s^n\|U_s\|_n^2+2s^n\langle A^nU_s,-U_s^3\rangle+\mathfrak{B}_ns^n\Big] \dif s+M_t
  \\ &\leq & \int_0^tC_{T,n}(1+s\|U_s\|_1^2)s^{n-1}\|U_s\|_n^2\dif s +\frac{\mathfrak{B}_n t^{n+1}}{n+1}
  +M_t-\frac{\gamma}{2}\langle M\rangle_t+\frac{\gamma}{2}\langle M\rangle_t
    \\ &\leq & C_{T,n}(\cE(1,t)+1) \int_0^t s^{n-1}\|U_s\|_n^2\dif s +C_n  t^{n+1}
  +M_t-\frac{\gamma}{2}\langle M\rangle_t+\gamma \cdot T\cdot \cE(t).
\end{eqnarray*}
Therefore,
\begin{eqnarray*}
\cE(n,t)-C_{T,n}(1+\cE(1,t))\cE(n-1,t)-C_n  t^{n+1}-\gamma\cdot  T \cdot \cE(t)  \leq M_t-\frac{\gamma}{2}\langle M\rangle_t.
\end{eqnarray*}
By the supermartingale  inequality, we have
\begin{eqnarray*}
 &&  \mP\left(\sup_{t\in[0,T]} \big(\cE(n,t)-C_{T,n}\cE(1,t)\cE(n-1,t)-C_n  t^{n+1}-\gamma\cdot  T \cdot \cE(t)  \big)\geq \rho\right)
  \\&&\leq e^{-\gamma \rho}.
\end{eqnarray*}
Since the inequality (\ref{2-2}) holds for $k\leq n-1$, using   similar arguments as that in Lemma \ref{2-1},
the inequality (\ref{2-2}) holds for $k=n.$
\end{proof}

\begin{lemma}
\label{8-1}
  For any $n, m,T>0$ and $0< s\leq T,$ there exists a positive constant $\lambda=\lambda_{n, m,T} $ such that
  \begin{eqnarray*}
    \mE \sup_{t\in [s,T]}\| U_t\|_n^m\leq C_{n,m,T}(s^{-\lambda}+1).
  \end{eqnarray*}
\end{lemma}
\begin{proof}
  By Lemma  \ref{1-8} and Lemma  \ref{4-1}, one sees that for some $\kappa,\gamma>0$
  \begin{eqnarray*}
   \mE \sup_{t\in [s,T]}\| U_t\|_n^m  \leq  C_{n,m,T}s^{-nm}\mE (\|U_s^\kappa\|^{2}+1)
  \leq  C_{n,m,T}s^{-nm}(s^{-\gamma}+1).
  \end{eqnarray*}
By setting   $\lambda=nm+\gamma$, we complete  the proof.
\end{proof}

\begin{lemma}\label{p-29}
For each  $\xi\in H$ and $0<s<t\leq T,$ we have the following pathwise estimates
\begin{eqnarray}
\label{p-38}
  && \|\cJ_{s,t}\xi\| \leq  \|\xi\|,
\quad
  \|\cK_{s,t}\xi\| \leq  \|\xi\|.
\end{eqnarray}
 Moreover, for each $\tau\leq T$ and $p\geq 1,$  there exists $C=C_{T,p}$ such that
\begin{eqnarray}
  \label{p-13} \mE \sup_{s<t \in [\tau,T]}\|\cJ_{s,t}^{(2)}(\xi,\xi')\|^p &\leq &  C\|\xi\|^p\|\xi'\|^p
\end{eqnarray}
\end{lemma}
\begin{proof}
  By (\ref{p-37}), for any $\xi\in H,$ we deduce that
\begin{eqnarray*}
 \dif    \|\cJ_{s,t}\xi\|^2& =& - 2 \langle A \cJ_{s,t}\xi, \cJ_{s,t}\xi\rangle \dif t
 +2 \langle \cJ_{s,t}\xi, \cJ_{s,t}\xi \rangle \dif t
 -\langle 6U_t^2\cJ_{s,t}\xi,\cJ_{s,t}\xi \rangle\dif t
 \\ &\leq &  -2\|\cJ_{s,t}\xi\|_1^2\dif t+2 \|\cJ_{s,t}\xi\|^2\dif t
\end{eqnarray*}
which implies
\begin{eqnarray}
\label{zp-1}
  \frac{\dif }{\dif t}\|\cJ_{s,t}\xi\|^2\leq 0
\end{eqnarray}
and
\begin{eqnarray}\label{pp-5}
 \|\cJ_{s,t}\xi\|^2+ 2 \int_s^t \|\cJ_{s,r}\xi\|_1^2 \dif \leq  \|\xi\|^2 e^{2(t-s)}.
\end{eqnarray}
By (\ref{zp-1}), one arrives at the first part of    (\ref{p-38}).
Moreover,   the second  part of    (\ref{p-38}) follows   by duality. It remains to prove  (\ref{p-13}).

For  any $U_0,\xi,\xi'\in H$,  the function $\varrho_t:=\cJ_{s,t}^{(2)}(\xi,\xi')\in H$ is the solution of
\begin{eqnarray*}
  \partial_t \varrho_t +A \varrho_t  -\varrho_t+3U_t^2\varrho_t+6U_t\cJ_{s,t}\xi \cJ_{s,t}\xi'=0,~~~~\quad \varrho(s)=0.
\end{eqnarray*}
Then
\begin{eqnarray*}
  \partial_t\|\varrho_t\|^2_H +2\langle A \varrho_t,\varrho_t \rangle -2 \|\varrho_t\|^2
  +\langle 3U_t^2\varrho_t+6U_t\cJ_{s,t}\xi \cJ_{s,t}\xi',\varrho_t\rangle =0.
\end{eqnarray*}
By Young's inequality and Sobolev embedding theorem, it yields
\begin{eqnarray*}
 \partial_t\|\varrho_t\|^2 & \leq &  -3 \langle U_t^2\varrho_t,\varrho_t \rangle - \langle 6U_t\cJ_{s,t}\xi \cJ_{s,t}\xi',\varrho_t\rangle+2\|\varrho_t\|^2
 \\ &\leq & 12 \int_{\mT}\big((\cJ_{s,t}\xi) (z)\big)^2\big((\cJ_{s,t}\xi') (z)\big)^2\dif z
 +2\|\varrho_t\|^2
 \\ &\leq & 12  \|\cJ_{s,t}\xi\|_{\infty}^2 \|\cJ_{s,t}\xi'\|^2+2\|\varrho_t\|^2
 \\ &\leq & 12  \|\cJ_{s,t}\xi\|_1^2\|\xi'\|^2+2\|\varrho_t\|^2.
\end{eqnarray*}
Thus,  by (\ref{pp-5}), we have
\begin{eqnarray*}
&& \|\cJ_{s,t}^{(2)}(\xi,\xi')\|^2
 \leq C_T  \int_s^t \|\cJ_{s,r}\xi\|_1^2  \dif r \|\xi'\|^2
\leq  C_T\|\xi\|^2\|\xi'\|^2
\end{eqnarray*}
which completes the proof of   (\ref{p-13}).
\end{proof}
\begin{lemma}\label{p-10}
 For any $p\geq 2,T\geq 0,$  there exists $C=C_{p,T}$ such that
 \begin{eqnarray*}
   \mE \sup_{t\in [T/2,T]} \|\partial_t K_{t,T}\xi\|^p_{H^{-2}} \leq C  \|\xi\|^p
 \end{eqnarray*}
\end{lemma}
\begin{proof}
Noting  that   $\rho_t^*=K_{t,T}\xi$ satisfies the following equation
 \begin{eqnarray*}
 && \partial_t \rho^* =  A\rho^*+(\nabla N(U_t))^{*}\rho^*=-(\nabla F(U_t))^*\rho^*,\quad \rho^*(T)=\xi,
\end{eqnarray*}
and
 \begin{eqnarray*}
  \|A\rho^*\|_{H^{-2}} & \leq  & \|\rho^*\|,
\\  \|(\nabla N(U(t)))^{*}\rho^*\|_{H^{-2}}
& \leq &   \sup_{\|\psi\|_{H^{2}}\leq 1}|\langle (\nabla N(U(t)))^{*}\rho^* , \psi\rangle|
\\ &\leq & \sup_{\|\psi\|_{H^{2}}\leq 1}|\langle\rho^* ,  (\nabla N(U(t)))\psi\rangle|
 \\ &  \leq  &  C  \sup_{\|\psi\|_{H^{2}}\leq 1} \|\rho^*\| \cdot \big[ \|U^2(t)\|+1\big]\cdot \|\psi\|_{\infty}
\\   &\leq &  C  \sup_{\|\psi\|_{H^{2}}\leq 1} \|\rho^*\| \cdot \big[ \|U^2(t)\|+1\big] \cdot  \|\psi\|_{1},
\end{eqnarray*}
by Lemmas \ref{4-1},  \ref{p-29}, we finish the proof of this lemma.
\end{proof}

For any $N\geq 1,$ define
\begin{eqnarray*}
  H_N:=span\{ e_k~:~ 0<|k|\leq N \},
\end{eqnarray*}
along with the associated projection operators
\begin{eqnarray*}
  P_N:H\rightarrow H_N \text{ the orthogonal projection onto $H_N$ }, ~~Q_N:=I-P_N.
\end{eqnarray*}

%
%With the same method in \cite{FGRT} and \cite{martin}, the following three  lemmas hold.
\begin{lemma}
  For every $p\geq 1,T>0,\delta>0$,  there exists    $N_*=N_*(p,T,\delta)$ such that  for any $N\geq N_*$ one has
  \begin{align}
  \label{6-2}
    \mE \|Q_N\cJ_{0,T}\|_{\cL\left( H,H\right)}^p \leq  \delta,\quad
 \mE \|\cJ_{0,T}Q_N \|_{\cL\left( H,H\right)}^p \leq  \delta
(\|U_0^p\|^2+1).
  \end{align}
  Here $\|\cdot \|_{\cL(X,Y)}$ denotes the operator norm of linear map between
  the given Hilbert spaces $X$ and $Y$.
\end{lemma}
\begin{proof}
For any $m\geq 1,$ by the  $It\^{o}$ formula in \cite[Theorem 7.7.5]{KS2012} and following similar arguments in the proof of \cite[Proposition 2.4.12]{KS2012}, it holds that
\begin{eqnarray*}
  && t^m\langle A\cJ_{0,t}\xi,  \cJ_{0,t}\xi\rangle =\int_0^t \Big[ ms^{m-1}\langle A\cJ_{0,s}\xi,  \cJ_{0,s}\xi\rangle+  2s^m \langle  A\cJ_{0,s}, \partial_s
   \cJ_{0,s}\xi\rangle \Big] \dif s
   \\ && = \int_0^t \Big[ ms^{m-1}\langle A\cJ_{0,s}\xi,  \cJ_{0,s}\xi\rangle+2s^m \langle  A\cJ_{0,s},  -A\cJ_{0,s}\xi +\cJ_{0,s}\xi-3U_s^2\cJ_{0,s}\xi\rangle \Big] \dif s
   \\ && = \int_0^t\Big[  (2s^m+ms^{m-1})\|\cJ_{0,s}\xi\|_1^2-2s^m\|\cJ_{0,s}\xi\|_2^2-6s^m\langle A\cJ_{0,s}\xi,U_s^2\cJ_{0,s}\xi \rangle \Big]\dif s
   \\ && \leq \int_0^t \Big[ (2s^m+ms^{m-1})\|\cJ_{0,s}\xi\|_1^2-2s^m\|\cJ_{0,s}\xi\|_2^2+6s^m\|\cJ_{0,s}\xi\|_2\| U_s \|_\infty^2 \| \cJ_{0,s}\xi\|\Big] \dif s
      \\ && \leq  \int_0^t \Big[ (2s^m+ms^{m-1})\|\cJ_{0,s}\xi\|_1^2-2s^m\|\cJ_{0,s}\xi\|_2^2+6s^m\big[\frac{1}{6}\|\cJ_{0,s}\xi\|_2^2+6\|U_s\|_1^4 \|\cJ_{0,s}\xi\|^2\big]\Big]\dif s
      \\ && \leq   \int_0^t \Big[(2s^m+ms^{m-1})\|\cJ_{0,s}\xi\|_1^2 +36s^m \|U_s\|_1^4\|\cJ_{0,s}\xi\|^2\Big] \dif s
      \\ &&\leq C_{T,m}\|\xi\|^2+ C_{T} \int_0^t  s^m \|U_s\|_1^4\|\xi\|^2  \dif s,
\end{eqnarray*}
where in the last inequality, we have used (\ref{pp-5}).
By the above inequality and  Lemma \ref{8-1}, there exists a $m>1$ such that
\begin{eqnarray}
\label{6-3}
  \mE \big( t^m\|\cJ_{0,t}\xi\|_1^2\big)^p \leq C_{T,m}\|\xi\|^{2p} ,\quad \forall t\in [0,T].
\end{eqnarray}
Fix this $m.$
Noting (\ref{6-3}) and  $\|Q_Nv\|\leq \frac{1}{N}\|v\|_1,\forall v\in V,$ we get
\begin{eqnarray*}
 \mE \|Q_N\cJ_{0,T}\xi\|^p\leq \frac{1}{N^p }\mE \|\cJ_{0,T}\xi\|_1^p \leq
 \frac{1}{N^p }\big(\mE \|\cJ_{0,T}\xi\|_1^{2p}\big)^{1/2}
 \leq  \frac{C_{T,m,p}\|\xi\|^p }{N^p \cdot T^{m/2}} ,
\end{eqnarray*}
which implies   the first part of (\ref{6-2}).

Now, we consider the second part of (\ref{6-2}).
For any $\xi\in H,$ let $\tilde{\xi}=Q_N\xi,\xi_t=\cJ_{0,t}\tilde{\xi}.$
Then  $\xi_t$ satisfies the following equation.
\begin{eqnarray}
\label{12-1}
\Bigg\{
\begin{split}
  & \partial_t  \xi_t=- A\xi_t +\xi_t-3U_t^2\xi_t
  \\ & \xi_0=\tilde{\xi}.
  \end{split}
\end{eqnarray}
Denote   $\xi_t^h=Q_N\xi_t,\xi_t^l=P_N\xi_t.$
One easily sees that
\begin{eqnarray*}
\Bigg\{
\begin{split}
  & \partial_t  \xi_t^h =- A\xi_t^h +\xi_t^h-Q_N(3U_t^2\xi_t)
  \\ & \xi_0^h=\tilde{\xi}.
  \end{split}
\end{eqnarray*}
and
\begin{eqnarray*}
  \partial_t\|\xi_t^h\|^2 &=&2  \langle \xi_t^h,  \partial_t\xi_t\rangle
   = 2\langle \xi_t^h,  - A\xi_t +\xi_t-3U_t^2\xi_t \rangle
   \\ &\leq & \big(2-2N^2)\|\xi_t^h\|^2+\|\xi_t^h\|^2+C\|\cJ_{0,t}\tilde{\xi}\|^2\|U_t\|_\infty^2
   \   \\ &\leq & (3-2N^2)\|\xi_t^h\|^2+C_T\|U_t\|_1^2\|\tilde{\xi}\|^2.
\end{eqnarray*}
%Thus,
%\begin{eqnarray*}
%  \partial_t\big(\|\xi_t^h\|^2e^{(2N^2-2)t}\big)
%  &\leq & C_Te^{(2N^2-2)t}\|U_t^2\|^2\|\xi\|^2.
%\end{eqnarray*}
By Gronwall  inequality, for any $s\leq t$, we have
\begin{eqnarray*}
\|\xi_t^h\|^2&\leq & \|\xi_s^h\|^2e^{-(2N^2-3)(t-s)} +e^{-(2N^2-3)t} \int_s^t \big(C_Te^{(2N^2-3)r}\|U_r\|_1^2\|\tilde{\xi}\|^2\big)\dif r
\\ &\leq & \|\xi_s^h\|^2e^{-(2N^2-3)(t-s)} +C_T \frac{1}{2N^2-3} \sup_{r\in [s,t]} \|U_r\|_1^2 \|\tilde{\xi}\|^2 .
\end{eqnarray*}
Thus,  for any $p\geq 2$, it holds that
\begin{eqnarray*}
\label{100-1}
\mE \|\xi_t^h\|^p&\leq & C_{T,p} \mE \|\xi_s^h\|^p e^{-(2N^2-3)(t-s)\cdot \frac{p}{2}} +C_{T,p} \frac{1}{2N^2-3} \mE \sup_{r\in [s,t]} \|U_r\|_1^p \|\tilde{\xi}\|^p.
\end{eqnarray*}
In the above inequality, let $s=\frac{t}{2}$.  By Lemma \ref{8-1}, for some $\gamma>0,$ we have
 \begin{eqnarray}
\label{100-1}
\mE \|\xi_t^h\|^p\leq  C_{T,p}\|\tilde{\xi}\|^p\big[e^{-\frac{N^2t}{4}}+\frac{1}{2N^2-3}t^{-\gamma}\big]
\end{eqnarray}
which yields
\begin{eqnarray}
\label{7-1}
 \mE \|\xi_T^h\|^p\leq \frac{\delta}{2} \|\xi\|^p
\end{eqnarray}
for $N$ big enough.

Now, we consider the estimate of   $\xi_T^l$.  For any  $0\leq t\leq T,$
one sees that $\xi_t^l$ satisfies the following equation
\begin{eqnarray*}
\Bigg\{
\begin{split}
  & \partial_t  \xi_t^l =- A\xi_t^l+\xi_t^l-P_N(3U_t^2\xi_t)
  \\ & \xi_t^l|_{t=0}=0.
  \end{split}
\end{eqnarray*}
We claim that  for any $\delta>0,$ there exists    $N_*=N_*(p,T,\delta)$ such that  for any $N\geq N_*$ one has
  \begin{align}
  \label{z-1}
 \mE \|\xi_T^\ell \|^p \leq  \frac{\delta}{2}
(\|U_0^p\|^2+1)\|\xi\|^p.
  \end{align}
 Once we have proved this,  combining (\ref{z-1})  with (\ref{7-1}), we
  obtain  the second part of (\ref{6-2}).

  Now we give a proof of  (\ref{z-1}).
Obviously, we have
\begin{eqnarray}
\label{6-4}
\begin{split}
  & \partial_t\|\xi_t^l\|^2 =  \langle \xi_t^l,  \partial_t\xi_t\rangle
\leq  -\langle\xi_t^l, 3U_t^2\xi_t \rangle=-\langle\xi_t^l, 3U_t^2(\xi_t^l+\xi_t^h) \rangle
  \\ & \leq | \langle \xi_t^l,  3U_t^2\xi_t^h\rangle|
    \leq  3\|U_t\|_\infty^2   \| \xi_t^l\| \|\xi_t^h\|.
    \end{split}
\end{eqnarray}
On the set   $\{ \|\xi_T^\ell \|\neq 0 \}$, we   define
\begin{eqnarray*}
  \tau=\sup \left\{t\in [0,T], \|\xi_t^\ell \|=0 \right\}.
\end{eqnarray*}
Hence, for $t\in (\tau,T]$, by (\ref{6-4}),  it holds that
\begin{eqnarray*}
 \partial_t \|\xi_t^l\| =\partial_t\sqrt{\|\xi_t^l\|^2}=\frac{ \partial_t\|\xi_t^l\|^2}{2 \sqrt{\|\xi_t^l\|^2}}
 \leq C\|U_t\|_\infty^2  \|\xi_t^h\|.
\end{eqnarray*}
Therefore,
\begin{eqnarray*}
  \|\xi_T^l\|\leq  \|\xi_\tau^l\|+C \int_\tau^T \|U_s\|_\infty^2  \|\xi_s^h\|\dif s\leq  C \int_0^T \|U_s\|_{\frac{1}{2}}^2  \|\xi_s^h\|\dif s.
\end{eqnarray*}
which implies
\begin{eqnarray}
\label{2000-3}
\begin{split}
 &  \mE \|\xi_T^l\|^{p} \leq C_{p}\mE \left(\int_0^t\|U_r\|_{\frac{1}{2}}^2 \|\xi_r^h\| \dif r\right)^{p}
+C_{p}\mE \left(\int_t^T\|U_r\|_{\frac{1}{2}}^2 \|\xi_r^h\| \dif r\right)^{p}
\\ & \leq C_{p}\mE \left(\int_0^t\|U_r\| \cdot \|U_r\|_1\cdot  \|\xi_r^h\| \dif r\right)^{p}
+C_{p}\mE \left(\int_t^T\|U_r\|_{\frac{1}{2}}^2 \|\xi_r^h\| \dif r\right)^{p}
\\ & :=I_1+I_2,
\end{split}
\end{eqnarray}
where  $t>0$ is a small parameter to be adjusted later.
As for $I_1,$ by (\ref{p-38}), Lemmas  \ref{2-1},\ref{p-29} and H\"older inequality, we have
\begin{align*}
  I_1 &\leq   C_{T,p}  \mE\Big[ \sup_{s\in [0,t]}\|U_s\|^p  \cdot  \big(  \int_0^t\|U_r\|_1  \dif r\big)^{p}\Big]\cdot  \|\xi\|^p
  \\ &\leq    C_{T,p}  \Big[ \mE \sup_{s\in [0,t]}\|U_s\|^{2p} \Big]^{1/2} \cdot  \Big[  \mE  \big(  \int_0^t\|U_r\|_1  \dif r\big)^{2p}\Big]^{1/2} \cdot \|\xi\|^p
 \\ & \leq  C_{T,p}   \Big[ \mE \sup_{s\in [0,t]}\|U_s\|^{2p} \Big]^{1/2} \cdot  \Big[ t^p \cdot   \mE  \big(\int_0^t\|U_r\|_1^2   \dif r\big)^{p}\Big]^{1/2}\cdot  \|\xi\|^p
 \\ &\leq  C_{T,p}t^{p/2}[\|U_0\|^{2p}+1]\|\xi\|^p.
\end{align*}
Setting   $t$ small enough, one arrives at that
\begin{eqnarray}
\label{2000-2}
\begin{split}
&I_1 \leq \frac{\delta}{4} (1+\|U_0^p\|^2)\|\xi\|^p.
  \end{split}
\end{eqnarray}
Fix this $t.$
Since
\begin{eqnarray*}
  && I_2\leq C_{T,p} \mE \left( \sup_{r\in [t,T]}\|U_r\|_{\frac{1}{2}}^{2p}\cdot  \big( \int_t^T\|\xi_r^h\| \dif r\big)^{p}\right)
  \\ &&\leq C_{T,p} \left( \mE  \sup_{r\in [t,T]}\|U_r\|_{\frac{1}{2}}^{4p} \right)^{1/2} \left(\mE \big( \int_t^T\|\xi_r^h\| \dif r\big)^{2p}\right)^{1/2}
    \\ &&\leq C_{T,p} \left( \mE  \sup_{r\in [t,T]}\|U_r\|_{\frac{1}{2}}^{4p} \right)^{1/2} \left(\mE \int_t^T\|\xi_r^h\|^{2p} \dif r\right)^{1/2},
\end{eqnarray*}
by (\ref{100-1}) and Lemma \ref{8-1},  we can   choose    $N$ big enough such that
\begin{eqnarray}
\label{20-1}
 I_2 \leq \frac{\delta}{4}\|\xi\|^p.
\end{eqnarray}
Combining (\ref{20-1}), (\ref{2000-2}) and (\ref{2000-3}), we obtain  the second part of (\ref{z-1}).
\end{proof}

 Using  the same method as \cite[Lemmas A.6,A.7]{FGRT} and \cite{martin}, by  Lemma \ref{p-29},  the following two    lemmas hold.

\begin{lemma}
\label{rr-1}
  For $0<s<t$, we have
  \begin{eqnarray*}
    \|\cA_{s,t}\|_{\cL\left(L^2([s,t],\mR^m),H\right)}\leq C\left(\int_s^t \| \cJ_{r,t}\|_{\cL\left( H,H\right)}^2\dif r\right)^{1/2}
  \end{eqnarray*}
   where  $C$ is a constant  independent of $s,t.$ Moreover, for any $\beta>0$, the following hold
  \begin{eqnarray*}
    \|\cA_{s,t}^* (\cM_{s,t}+I\beta)^{-1/2}\|_{\cL\left(H, L^2([s,t],\mR^m)\right)}& \leq & 1,
    \\
    \| (\cM_{s,t}+I\beta)^{-1/2}\cA_{s,t}\|_{\cL\left( L^2([s,t],\mR^m),H\right)}& \leq & 1,
        \\
    \| (\cM_{s,t}+I\beta)^{-1/2}\|_{\cL\left( H,H\right)}& \leq & \beta^{-1/2},
    \\
    \| (\cM_{s,t}+I\beta)^{-1}\|_{\cL\left( H,H\right)}& \leq & \beta^{-1}.
  \end{eqnarray*}

\end{lemma}
Observe that for $\tau\leq t$
\begin{eqnarray*}
  \cD_\tau^j\cJ_{s,t}\xi=
  \left\{\begin{split}
    & \cJ_{\tau,t}^{(2)}(G\theta_j,\cJ_{s,\tau}\xi)\quad \text{  if } s\leq \tau,
    \\
    & \cJ_{s,t}^{(2)}(\cJ_{\tau,s}G\theta_j,\xi) \quad \text{  if } s>\tau.
  \end{split}
  \right.
\end{eqnarray*}
\begin{lemma}
\label{a-2}
  For any $\xi\in H,0\leq s\leq t\leq T$ and $p\geq 1$ we have the bounds
  \begin{eqnarray*}
    \mE\|\cD_\tau^j\cJ_{s,t}\xi\|^p &\leq & C\|\xi\|^p,
    \\    \mE\|\cD_\tau^j\cA_{s,t}\|^p_{\cL\left( L^2([s,t],\mR^m),H\right)} &\leq & C,
      \\    \mE\|\cD_\tau^j\cA_{s,t}^*\|^p_{\cL\left( H, L^2([s,t],\mR^m)\right)} &\leq & C,
  \end{eqnarray*}
  where $C=C_{T, p}$.
\end{lemma}

\section{Spectral properties of Malliavin matrix  $\cM$}
For any $\alpha>0, N\in \mN,$ we define
\begin{eqnarray*}
  \cS_{\alpha,N}:=\{\phi\in H:\|P_N\phi\|^2\geq \alpha \|\phi\|^2\}.
\end{eqnarray*}

The aim of this section is to prove the following result:

\begin{thm}\label{p-9}
For any $N\geq 1,\alpha\in (0,1]$ and $T>0$, there exists a positive constant $\eps^*=\eps^*(\alpha,N,T)>0,$ such that, for any $n\geq 0,$ and $\eps\in (0,\eps^*]$, there exists a measurable set $\Omega_\eps=\Omega_\eps(\alpha,N,T)\subseteq \Omega$ satisfying
  \begin{eqnarray}\label{p-54}
    \mP(\Omega_\eps^c )\leq r(\eps ),
  \end{eqnarray}
    where $r=r(\alpha,N,T):(0,\eps^*]\rightarrow (0,\infty)$ is a non-negative, decreasing function with $\lim_{\eps \rightarrow 0}r(\eps)=0,$ and on the set $\Omega_\eps,$
  \begin{eqnarray}\label{p-55}
\inf_{\phi\in \cS_{\alpha,N}}\frac{\langle \cM_{0,T}\phi,\phi\rangle}{\|\phi\|^2}\geq \eps.
  \end{eqnarray}
\end{thm}

In order to prove this theorem, we  show   the   details of  Lie bracket computations in subsection \ref{pp-13},
demonstrate   Proposition \ref{p-53} in subsection \ref{p-7}
and  Proposition \ref{h-5} in subsection \ref{p-8}. Finally, give a proof the  of Theorem  \ref{p-9}
 in subsection \ref{p-59}.

\subsection{Details of  Lie bracket computations}
\label{pp-13}

For any Fr\'echet differentiable $E_1,E_2:H\rightarrow H,$
\begin{eqnarray*}
  [E_1,E_2](u):=\nabla E_2(u)E_1(u)-\nabla E_1(u) E_2(u).
\end{eqnarray*}
This operator $[E_1,E_2]$ is referred   as the  Lie bracket of two ``vector fileds" $E_1,E_2.$
For any $k,\ell,j\in \mZ_*, m,m',m''\in \{0,1\}$,  by calculating, for any $u=u(z)\in H$
\begin{eqnarray}
\nonumber  I_k^m(u)&:=& [ F(u), \cos(kz+\frac{\pi}{2}m)]
 \\ \nonumber &=&A\cos(kz+\frac{\pi}{2}m)+3u^2 \cos(kz+\frac{\pi}{2}m)- \cos(kz+\frac{\pi}{2}m),
  \\ \nonumber \cJ_{k,\ell}^{m,m'}(u)&:=&  -[[ F(u), \cos(kz+\frac{\pi}{2}m) ],\cos(\ell z+\frac{\pi}{2}m')]
  \\ \nonumber &=& 6u  \cos(kz+\frac{\pi}{2}m)\cos(\ell z+\frac{\pi}{2}m').
\\ \nonumber  K_{k,\ell,j}^{m,m',m''}(u)&:=&-[\cJ_{k,\ell}^{m,m'},\cos(jz+\frac{\pi}{2}m'') ]
  \\  \label{pp-8} &=& 6\cos(kz+\frac{\pi}{2}m)\cos(\ell z+\frac{\pi}{2}m')\cos(jz+\frac{\pi}{2}m'').
\end{eqnarray}
Therefore, for any $k,\ell,j \in \mZ,$ we have
\begin{eqnarray}
\label{pp-7}
\begin{split}
& \cos((k+\ell+j)z)=\sum_{m,m',m''\in \{0,1\}}C_1^{m,m',m''} K_{k,\ell,j}^{m,m',m''}(u)
\\
 & \sin((k+\ell+j)z)=\sum_{m,m',m''\in \{0,1\}}C_2^{m,m',m''} K_{k,\ell,j}^{m,m',m''}(u)
\end{split}
\end{eqnarray}
where   $C_i^{m,m',m''}, i=1,2$  are some constants depending  on $k,\ell,j, m,m',m''.$

\subsection{Quadratic forms: lower bounds}\label{p-7}
Denote
\begin{eqnarray*}
  \langle \cQ_N\phi,\phi\rangle:= \sum_{0\leq |k|\leq N}|\langle \phi, e_k\rangle|^2
\end{eqnarray*}
One easily sees that  the following Proposition holds.
\begin{proposition}\label{p-53}
  Fix  any integer $N\in \mN,$ then for any $U\in H$ and $\alpha\in (0,1]$,
  \begin{eqnarray*}
  \langle \cQ_N \phi,\phi\rangle \geq \frac{\alpha}{2}\|\phi\|^2
  \end{eqnarray*}
   holds  for every $\phi \in\cS_{\alpha,N}.$
\end{proposition}
\subsection{Quadratic forms: upper bounds}\label{p-8}

The aim of this subsection is  to  prove  the following proposition.
\begin{proposition}\label{h-5}
Fix $T>0,$ for  any $ N\geq 1,\alpha\in (0,1]$,   there are positive constant $q_1=q_1(\alpha,N,T),q_2=q_2(\alpha,N,T)$ such that the following holds.
There exists a positive constant $\eps^*=\eps^*(\alpha,N,T)>0,$ such that, for any   $\eps\in (0,\eps^*]$, there exist a measurable set $\Omega_\eps=\Omega_\eps(\alpha,N,T)\subseteq \Omega$  and positive  constants  $C_1=C_1(\alpha,N,T),C_2=C_2(\alpha,N,T)$ such that
  \begin{eqnarray*}
    \mP((\Omega_\eps)^c )\leq C_1\eps ^{q_1}
  \end{eqnarray*}
  and on the set  $\Omega_\eps$ one has
  \begin{eqnarray*}
\langle \cM_{0,T}\phi,\phi\rangle \leq \eps \|\phi\|^2 ~\Rightarrow~ \langle \cQ_N \phi,\phi\rangle \leq   C_2 \eps^{q_2}\|\phi\|^2
  \end{eqnarray*}
  which is valid for any  $\phi \in\cS_{\alpha,N}.$
\end{proposition}

\def\@captype{figure}
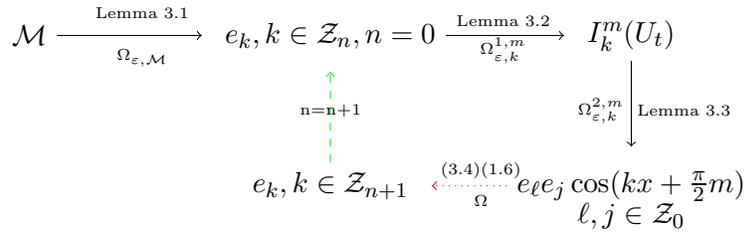
\begin{figure*}[h]
  \centering
\begin{tikzpicture}[shorten >=5pt,node distance=8cm]
% \draw  (0,0) rectangle (4,4);
\node(fs) at (0,4)   {$\cM$};
\node(ff) at (4,4)   {$e_k, k\in \cZ_n, n=0$};
\node(sf) at (8,4) {$I_k^m(U_t)$};
\node(st) at (8,2) {$e_\ell e_j\cos(kx+\frac{\pi}{2}m)$};
\node(sts) at (8,1.6) {$\ell,j\in\cZ_0$};
\node(ft) at (4,2) {$e_k, k\in \cZ_{n+1}$};
%  \text at (12,4)  {$[Y_k^m(U),\sigma_{\ell}^{m'}]$ };
% \node(fz) at (4,0) {$[\cY_k^m(U),\sigma_{\ell}^{m'}]$};
% \node(ez) at (8,0) {$\cY_k^m(U)$};
% \node(tz) at (12,0) {  $\psi_{k}^m,k\in \cZ_{2n+1}$};
% \node(zt) at (0,2) {$\sigma_{k}^m,k\in\cZ_{2n+2}$};
 \draw[->](fs) -- (ff);
  \draw[->](ff) -- (sf);

  \node at (6.3,3.8){\tiny  $\Omega_{\eps,k}^{1,m}$};
 \node at (6.3,4.2){\tiny  Lemma \ref{p-16} };
  \node at (7.6,3){\tiny  $\Omega_{\eps,k}^{2,m}$};
 \node at (8.7,3){\tiny  Lemma \ref{p-17} };
   \draw[->](sf) -- (st);
      \draw[dotted, red, ->](st) -- (ft);
        \node at (6,1.8){\tiny  $\Omega$};
 \node at (6,2.2){\tiny  (\ref{pp-7})(\ref{pp-10}) };
       \draw[dashed,green, ->](ft) -- (ff);
        \node at (4,3 ){\tiny  n=n+1 };

        \node at (1.5,4.3){\tiny  Lemma \ref{h-10} };
        \node at (1.5,3.7){\tiny  $\Omega_{\eps,\cM}$};
\end{tikzpicture}
\caption{\footnotesize  An illustration of the structure of the lemmas that leads to the proof of Proposition \ref{h-5}. In this figure, $m \in \{0,1\}, \ell\in \cZ_0.$   The  solid arrows  indicate that if one term is ``small" then the other one ``small" on a set of large measure(displayed below  or left of the arrow), where the meaning of ``smallness" is made precisely in each lemma. The  dashed  arrow  with color green  shows that the process is iterative.
 The  dotted arrow with  color  red   signify that  the new element is generated as    a linear combination of elements from the previous  actually.}
\label{t1}
\end{figure*}

In the  Figure \ref{t1}, we   give an illustration of the structure of  lemmas in this subsection   that lead to a  proof of Proposition \ref{h-5}.

\begin{lemma}\label{h-10}
For any $0<\eps<\eps_0(T,\cE_0)$,  there exist a set $\Omega_{\eps,\cM} $
and a constant $C=C_T $ with
\begin{eqnarray*}
  \mP(\Omega_{\eps,\cM} ^c)\leq C \eps
\end{eqnarray*}
such that on the set $\Omega_{\eps,\cM}  $
  \begin{eqnarray}\label{p-52}
  \begin{split}
 &\langle \cM_{0,T}\phi,\phi\rangle \leq \eps \|\phi\|^2
\\ & \Rightarrow~ \sup_{t\in [T/2,T]}| \langle \cK_{t,T}\phi,e_\ell \rangle| \leq   \eps^{1/8}\|\phi\|
%\\ & \text{ and } \sup_{t\in [T/2,T]}| \langle \cK_{t,T}\phi,a \rangle| \leq   \eps^{1/8}\|\phi\|^2,  \sup_{t\in [T/2,T]}| \langle \cK_{t,T}\phi,b \rangle| \leq   \eps^{1/8}\|\phi\|^2
\end{split}
  \end{eqnarray}
for each $\ell  \in \cZ_0 $ and $\phi\in H.$
\end{lemma}
\begin{proof}
Note  that
\begin{eqnarray*}
\langle \cM_{0,T}\phi,\phi\rangle=\sum_{\ell \in \cZ_0}(\beta_\ell)^2\int_0^T \langle e_\ell,\cK_{r,T}\phi\rangle^2\dif r
%+ q_a^2\int_0^T \langle a,\cK_{r,T}\phi\rangle^2\dif r
%+ q_b^2\int_0^T \langle b,\cK_{r,T}\phi\rangle^2\dif r.
\end{eqnarray*}
Define the function $  g_{\phi}(\cdot):[T/2,T]\rightarrow \mR^{+},$
\begin{eqnarray*}
  g_{\phi}(t):=\sum_{\ell \in \cZ_0}(\beta_\ell)^2\int_0^t \langle e_\ell,\cK_{r,T}\phi\rangle^2\dif r
 % + q_a^2\int_0^t \langle a,\cK_{r,T}\phi\rangle^2\dif r
% + q_b^2\int_0^t \langle b,\cK_{r,T}\phi\rangle^2\dif r,
\end{eqnarray*}
then
\begin{eqnarray*}
  g_{\phi}'(t)& =& \sum_{\ell \in  \cZ_0 }(\beta_\ell)^2\langle e_\ell,\cK_{t,T}\phi\rangle^2
  % + q_a^2 \langle a,\cK_{t,T}\phi\rangle^2
% + q_b^2 \langle b,\cK_{t,T}\phi\rangle^2,
\\   g_{\phi}^{''}(t)&=&2\sum_{\ell \in  \cZ_0 }
 \beta_\ell^2 \langle e_\ell,\cK_{t,T}\phi\rangle \langle e_\ell,\partial_t \cK_{t,T}\phi\rangle
%+2 q_a^2 \langle a,\cK_{t,T}\phi\rangle  \langle a,\partial_t \cK_{t,T}\phi\rangle
%+ 2q_b^2 \langle b,\cK_{t,T}\phi\rangle \langle b,\partial_t\cK_{t,T}\phi\rangle.
\end{eqnarray*}
Let
\begin{eqnarray*}
\Omega_{\eps,\cM}&=& \bigcap_{\phi\in H,\|\phi\|=1} \Big\{\sup_{t\in [T/2,T]} |g_{\phi}(t)|\geq \eps \text{  or } \sup_{t\in [T/2,T]} |g_{\phi}'(t)|\leq \eps^{1/4} \Big\}.
\end{eqnarray*}

It is obvious that (\ref{p-52}) holds  on $\Omega_{\eps,\cM}$.
Setting  $\alpha=1$ in \cite[Lemma 6.2]{FGRT}, by  Lemmas \ref{p-29}, \ref{p-10},
one arrives at that
\begin{eqnarray*}
 \mP\left(\Omega_{\eps,\cM}^c\right)&\leq & \mP\left(\bigcup_{\phi\in H,\|\phi\|=1} \Big\{\sup_{t\in [T/2,T]} |g_{\phi}(t)|\leq \eps \text{  and } \sup_{t\in [T/2,T]} |g_{\phi}'(t)|\geq \eps^{1/4} \Big\}\right)
 \\ &\leq & C\eps \sum_{\ell \in \cZ_0 }(\beta_\ell)^4  \mE\[\sup_{t\in [T/2,T], \|\phi\|=1}|\langle e_\ell,\cK_{t,T}\phi\rangle \langle e_\ell,\partial_t \cK_{t,T}\phi\rangle |^2\]
 \\ &\leq &C\eps,
\end{eqnarray*}
which completes the proof of this lemma.
\end{proof}

\begin{lemma}\label{p-16}
  Fix  $k\in \mZ, m\in \{0,1\}.$ For any $0<\eps<\eps_0(T)$,  there  exist a set $\Omega_{\eps,k}^{1,m}$ and $C=C_{k,m,T}$ with
  \begin{eqnarray*}
   \mP((\Omega_{\eps,k}^{1,m}) ^c)\leq C\eps,
  \end{eqnarray*}
  such that on the set  $\Omega_{\eps,k}^{1,m},$   it holds that for any $\phi\in H$
  \begin{eqnarray}\label{pp-6}
    \sup_{t\in [T/2,T]} |\langle \cK_{t,T}\phi,\cos(kx+\frac{\pi}{2}m)\rangle|\leq \eps \|\phi\| \Rightarrow  \sup_{t\in [T/2,T]} |\langle \cK_{t,T}\phi,I_k^m(U_t))\rangle|\leq \eps^{1/10}\|\phi\|.
  \end{eqnarray}
\end{lemma}
\begin{proof}
  Define
   $ g_\phi(t):=\langle \cK_{t,T}\phi,\cos(kx+\frac{\pi}{2}m) \rangle, \forall t\in [0,T] $.  Observing   (\ref{p-3}), one has
   \begin{eqnarray*}
   g'_\phi(t)&=& \langle \cK_{t,T}\phi,[F(U_t),\cos(kz+\frac{\pi}{2}m)]\rangle
   \\ &=& \langle \cK_{t,T}\phi,I_k^m(U_t)\rangle.
   \end{eqnarray*}
Let $\alpha=\frac{1}{4}$  and define
\begin{eqnarray*}
\Omega_{\eps,k}^{1,m}&=& \bigcap_{\phi\in H,\|\phi\|=1} \Big\{\sup_{t\in [T/2,T]} |g_{\phi}(t)|\geq \eps \text{   ~or  } \sup_{t\in [T/2,T]} |g_{\phi}'(t)|\leq \eps^{\alpha/2(1+\alpha)} \Big\}.
\end{eqnarray*}
Then on  $\Omega_{\eps,k}^{1,m},$ (\ref{pp-6}) holds.
By \cite[Lemma 6.2]{FGRT}, it holds that
\begin{eqnarray}
 \nonumber \mP\left((\Omega_{\eps,k}^{1,m})^c\right)&\leq & \mP\left(\bigcup_{\phi\in H,\|\phi\|=1} \Big\{\sup_{t\in [T/2,T]} |g_{\phi}(t)|\leq \eps \text{  and } \sup_{t\in [T/2,T]} |g_{\phi}'(t)|\geq \eps^{\alpha/2(1+\alpha)} \Big\}\right)
 \\ \label{pp-15}&\leq & C\eps \mE\[\sup_{\phi, \|\phi\|=1}\|g_{\phi}'\|_{C^\alpha[T/2,T]}^{2/\alpha}\]
\end{eqnarray}

Note that
\begin{eqnarray*}
   g'_\phi(t)&=&  \langle \cK_{t,T}\phi,I_k^m(U_t)\rangle
  = \langle \cK_{t,T}\phi,  Af +3U_t^2f- f \rangle,
 \end{eqnarray*}
 where $f:\mT\rightarrow\mR$ is a function given by $f(z)=\cos(kz+\frac{\pi}{2}m ).$
We have
\begin{eqnarray*}
&& \|g_{\phi}'\|_{C^\alpha[T/2,T]}
  \\ & &\leq  C \sup_{t\in [T/2,T]}|\partial_t  \langle \cK_{t,T}\phi,Af \rangle|+ C \|\langle \cK_{t,T}\phi,U_t^2 f \rangle\|_{C^\alpha[T/2,T]}
 +  C \sup_{t\in [T/2,T]} |\partial_t  \langle \cK_{t,T}\phi, f \rangle|
\\ &&\leq  C \!\sup_{t\in [T/2,T]} \!|\partial_t  \langle \cK_{t,T}\phi,Af \rangle|+   C\!\sup_{t\in [T/2,T]} \| \partial_t \cK_{t,T}\phi\|_{H^{-2}} \!\sup_{t\in [T/2,T]}\| U_t^2f   \|_{2}
\\&&\quad +  C\sup_{t\in [T/2,T]}  \| \cK_{t,T}\phi\| \cdot \|  U_t^2f\|_{C^\alpha([T/2,T],H)}+  C \sup_{t\in [T/2,T]} |\partial_t  \langle \cK_{t,T}\phi, f\rangle|
\end{eqnarray*}
By (\ref{p-1}), Lemma \ref{8-1} and
\begin{eqnarray*}
  \frac{\|(U_{t_1}^2-U_{t_2}^2)f\|}{|t_1-t_2|^\alpha}\leq (\|U_{t_1}f\|+\|U_{t_2}f\|)\cdot \frac{\|U_{t_1}-U_{t_2}\|}{|t_1-t_2|^\alpha}\leq (\|U_{t_1}\|+\|U_{t_2}\|)\cdot \frac{\|U_{t_1}-U_{t_2}\|}{|t_1-t_2|^\alpha},
\end{eqnarray*}
we obtain
\begin{eqnarray*}
\mE \|  U_t^2f\|_{C^\alpha([T/2,T],H)}^p \leq C_{T,p}, \quad \forall p\geq 1.
\end{eqnarray*}
Therefore, by Lemmas \ref{8-1},\ref{p-29},\ref{p-10} and the above equality,  one arrives at that
\begin{eqnarray*}
\mE\[\sup_{\phi:\|\phi\|=1}\|g_{\phi}'\|_{C^\alpha[T/2,T]}^{2/\alpha}\]
&&\leq C_T.
\end{eqnarray*}
Combining the above inequality with (\ref{pp-15}), the proof   is completed.
\end{proof}

\begin{lemma}\label{p-17}
  Fix any $k\in \mZ, m\in \{0,1\}.$ For any $0<\eps<\eps_0(T)$,  there  exist a set $\Omega_{\eps,k}^{2,m}$ and $C=C_{T,k}$ with
  \begin{eqnarray}\label{pp-20}
   \mP((\Omega_{\eps,k}^{2,m}) ^c)\leq C\eps^{1/27},
  \end{eqnarray}
  such that on the set  $\Omega_{\eps,k}^{2,m},$ it holds that   for any $\phi\in H$
  \begin{eqnarray}
  \label{pp-18}
  \begin{split}
   &  \sup_{t\in [T/2,T]} |\langle \cK_{t,T}\phi,I_k^m(U_t)\rangle|\leq \eps \|\phi\|
    \\ &\Rightarrow  \sup_{\ell,j \in \cZ_0}\sup_{t\in [T/2,T]} |\beta_\ell \beta_j|\cdot |\langle \cK_{t,T}\phi,e_\ell e_j
   \cos(kx+\frac{\pi}{2}m)  \rangle| \leq \eps^{1/9}\|\phi\|.
   \end{split}
  \end{eqnarray}
\end{lemma}

\begin{proof}
Denote $L_t=U_0+ \int_0^t F(U_s)\dif s,$ then we have
\begin{eqnarray*}
  I_k^m(U_t) &=& A\cos(kz+\frac{\pi}{2}m)-\cos(kz+\frac{\pi}{2}m)+3U_t^2 \cos(kz+\frac{\pi}{2}m)
  \\ &=&  A\cos(kz+\frac{\pi}{2}m)-\cos(kz+\frac{\pi}{2}m)+3\bigg(L_t+\sum_{k\in \cZ_0}\beta_ke_kW_k(t)\bigg)^2 \cos(kz+\frac{\pi}{2}m)
  \\ &=& A\cos(kz+\frac{\pi}{2}m) -\cos(kz+\frac{\pi}{2}m)
  \\ && +3\bigg(L_t^2+2L_t \sum_{\ell \in \cZ_0}\beta_\ell e_\ell W_\ell(t)+  \sum_{\ell,j \in \cZ_0 }\beta_\ell e_\ell \beta_j e_j W_\ell(t)W_j(t)  \bigg)  \cos(kz+\frac{\pi}{2}m).
\end{eqnarray*}
Therefore,
\begin{eqnarray*}
 &&  \langle \cK_{t,T}\phi,  I_k^m(U_t) \rangle
  \\ &&= \langle \cK_{t,T}\phi,  A\cos(kz+\frac{\pi}{2}m)- \cos(kz+\frac{\pi}{2}m) +3 L_t^2 \cos(kz+\frac{\pi}{2}m) \rangle
  \\ &&\quad +6 \sum_{\ell \in \cZ_0} \langle \cK_{t,T}\phi, L_t\beta_\ell e_\ell \cos(kz+\frac{\pi}{2}m)\rangle W_\ell (t)
  \\ &&\quad +3 \sum_{\ell,j \in \cZ_0  }\langle \cK_{t,T}\phi,\beta_\ell e_\ell \beta_j e_j
   \cos(kz+\frac{\pi}{2}m)  \rangle  W_\ell(t)W_j(t)
   \\ &&:=A_0(t)+\sum_{\ell \in \cZ_0} A_\ell W_\ell (t)+\sum_{\ell,j \in \cZ_0 }A_{\ell,j}W_\ell(t)W_j(t).
\end{eqnarray*}
By Lemmas \ref{8-1}-\ref{p-10},   for any $T,p>0,$ we have
\begin{eqnarray}
\nonumber   && \mE \[\sup_{s\neq t\in [T/2,T]}\Big|\frac{|A_0(t)-A_0(s)|}{|t-s|} +\sum_{\ell \in \cZ_0} \frac{|A_\ell(t)-A_\ell(s)|}{|t-s|}+\sum_{\ell,j \in \cZ_0 }\frac{|A_{\ell,j}(t)-A_{\ell,j}(s)|}{|t-s|}\Big|^p \]
 \\ \label{p-22} && \leq C_{T,p}.
\end{eqnarray}

Define
\begin{eqnarray*}
  \cN_1(\phi)&:=&\sup_{s\neq t\in [T/2,T]}\Big|\frac{|A_0(t)-A_0(s)|}{|t-s|} +\sum_{\ell \in \cZ_0} \frac{|A_\ell (t)-A_\ell (s)|}{|t-s|}+\sum_{\ell,j \in \cZ_0 }\frac{|A_{\ell,j}(t)-A_{\ell,j}(s)|}{|t-s|}\Big|,
  \\
    \cN_0(\phi)&:=&\sup_{s\neq t\in [T/2,T]}\Big| |A_0(t)| +\sum_{\ell \in \cZ_0} |A_\ell(t)|+\sum_{\ell,j \in \cZ_0 }|A_{\ell,j}(t)|\Big|.
\end{eqnarray*}
By \cite[Theorem 6.4]{FGRT}, there exists a set $\Omega^{\#}_\eps$ such that
\begin{eqnarray}\label{pp-21}
\mP( (\Omega^{\#}_\eps)^c)\leq C\eps,
\end{eqnarray}
and
on $\Omega^{\#}_\eps$, we have
\begin{eqnarray}\label{pp-17}
   \sup_{t\in [T/2,T]} |\langle \cK_{t,T}\phi,I_k^m(U)\rangle|\leq \eps \Rightarrow
   \left\{
   \begin{split}
     &  \text{either } \cN_0(\phi)\leq \eps^{1/9},
     \\
     &\text{or  } \cN_1(\phi)\geq \eps^{-1/27}.
   \end{split}
   \right.
\end{eqnarray}
Therefore, we obtain
\begin{eqnarray}\label{pp-19}
   \sup_{t\in [T/2,T]} |\langle \cK_{t,T}\phi,I_k^m(U)\rangle|\leq \eps \Rightarrow
  \cN_0(\phi)\leq \eps^{1/9}
\end{eqnarray}
on a  set
\begin{eqnarray*}
\Omega_{\eps,k}^{2,m}: =\Omega^{\#}_\eps \cap \cap_{\phi \in H, \|\phi\|=1 } \{\cN_1(\phi)<  \eps^{-1/27}  \}.
\end{eqnarray*}
Combining (\ref{pp-19})  with   the following  fact
 \begin{eqnarray*}
|\langle \cK_{t,T}\phi,\beta_\ell e_\ell \beta_j e_j
   \cos(kz+\frac{\pi}{2}m) \rangle|
  = |\beta_\ell \beta_j|\cdot |\langle \cK_{t,T}\phi,e_\ell e_j
   \cos(kz+\frac{\pi}{2}m)  \rangle|,
\end{eqnarray*}
one arrives at  (\ref{pp-18}). The desired result
(\ref{pp-20})  is implied  by
 (\ref{p-22}) and  (\ref{pp-21}).

\end{proof}

\begin{lemma}\label{h-11}
For any $n \in \mN$, and $ q_{n},C_{n}>0,$ there exist   $p_{n+1}, q_{n+1},C_{n+1}>0$, a constant   $C=C(n,T)$,  and   a set $\Omega_{\eps,n}$  with
  \begin{eqnarray*}
   \mP(\Omega_{\eps,n} ^c)\leq C \eps^{p_{n+1}},
  \end{eqnarray*}
  such that on the set  $\Omega_{\eps,n},$ it holds
    \begin{eqnarray*}
    && \sum_{k \in \cZ_{n},  }\sup_{t\in [T/2,T]} |\langle \cK_{t,T}\phi,e_k\rangle| \leq C_{n}\eps^{q_{n}} \|\phi\|
    \\ && \Rightarrow   \sum_{k  \in \cZ_{n+1} }\sup_{t\in [T/2,T]} |\langle \cK_{t,T}\phi,e_k\rangle|\leq C_{n+1}\eps^{q_{n+1}} \|\phi\|.
  \end{eqnarray*}
\end{lemma}
\begin{proof}
For any $n\geq 0,$ by Hypothesis  \ref{pp-2} and  the definition of  $\cZ_n$, one sees that
  \begin{eqnarray}
  \label{pz-1}
     \forall k \in\cZ_{n} \Rightarrow  -k\in \cZ_n.
  \end{eqnarray}
 Thus, on the set $\{\sum_{k \in \cZ_{n} }\sup_{t\in [T/2,T]} |\langle \cK_{t,T}\phi,e_k\rangle| \leq C_{n}\eps^{q_{n}} \|\phi\|\},$ it holds that
  \begin{eqnarray*}
 \sup_{t\in [T/2,T], k\in \cZ_n,  m\in \{0,1\}} |\langle \cK_{t,T}\phi,\cos(kz+\frac{\pi}{2}m) \rangle| \leq C_{n}\eps^{q_{n}} \|\phi\|.
  \end{eqnarray*}

   By Lemma \ref{p-16},
for  any $k\in \cZ_n, m\in \{0,1\},$  there  exist a set $\Omega_{\eps,k}^{1,m}$, $C=C_{k,m,T}$ and  $ p_n'>0$ with
  \begin{eqnarray}
  \label{m-1}
   \mP((\Omega_{\eps,k}^{1,m}) ^c)\leq C\eps^{p_n'},
  \end{eqnarray}
  such that on the set  $\Omega_{\eps,k}^{1,m},$   it holds that
  \begin{eqnarray}\label{pp-12}
   \nonumber  && \sup_{t\in [T/2,T]} |\langle \cK_{t,T}\phi,\cos(kz+\frac{\pi}{2}m)\rangle|\leq C_n\eps^{q_n} \|\phi\|
    \\ && \Rightarrow  \sup_{t\in [T/2,T]} |\langle \cK_{t,T}\phi,I_k^m(U_t))\rangle|\leq C_{n+1}'\eps^{q_{n+1}'}\|\phi\|
  \end{eqnarray}
for some $C_{n+1}', q_{n+1}'>0. $

    By Lemma \ref{p-17}, for any $k\in \cZ_n, m\in \{0,1\}, $ there exist $p_{n}^{''}, C_{n+1}, q_{n+1}$  and  a set   $\Omega_{\eps,k}^{2,m}$ such that  on $\Omega_{\eps,k}^{2,m}$,
  \begin{eqnarray}
   \nonumber    && \sup_{t\in [T/2,T]} |\langle \cK_{t,T}\phi,I_k^m(U) \rangle|\leq C_{n+1}' \eps^{q_{n+1}'} \|\phi\|
     \\ \label{p-11} && \Rightarrow  \sup_{t\in [T/2,T],\ell,j  \in\cZ_0} |\langle \cK_{t,T}\phi,e_\ell e_j
   \cos(kz+\frac{\pi}{2}m) \rangle|\leq C_{n+1}\eps^{q_{n+1}}\|\phi\|,
  \end{eqnarray}
  and
  \begin{eqnarray}
  \label{m-2}
    \mP(  (\Omega_{\eps,k}^{2,m})^c )\leq  C\eps^{p_{n}^{''}}.
  \end{eqnarray}

Let
  \begin{eqnarray*}
\Omega_{\eps,n}=\cap_{k  \in\cZ_{n}, m\in\{0,1\} }  \Omega_{\eps,k}^{1,m}\cap \Omega_{\eps,k}^{2,m} .
  \end{eqnarray*}
By (\ref{pp-12}) and (\ref{p-11}),  on the  set
$\Omega_{\eps,n}$,  it holds that
 \begin{eqnarray*}
 \begin{split}
     & \sum_{k \in \cZ_{n},  }\sup_{t\in [T/2,T]} |\langle \cK_{t,T}\phi,e_k\rangle| \leq C_{n}\eps^{q_{n}} \|\phi\|
    \\    \!\Rightarrow\!\!\!\!
   & \sup_{t\in [T/2,T]} \sup_{\mbox{\tiny$\begin{array}{c}
k\in \cZ_n,\ell,j \in \cZ_0  \\
m\in \{0,1\}
\end{array}$}}
|\langle \cK_{t,T}\phi,e_{\ell} e_j
   \cos(kz+\frac{\pi}{2}m) \rangle|\leq C_{n+1}\eps^{q_{n+1}}\|\phi\|
   \end{split}
 \end{eqnarray*}
for some  $C_{n+1},q_{n+1}>0.$
Since  (\ref{pz-1}) holds for $n=0,$
 on the  set
$\Omega_{\eps,n}$,  it also  holds that
 \begin{eqnarray*}
 \begin{split}
     & \sum_{k \in \cZ_{n},  }\sup_{t\in [T/2,T]} |\langle \cK_{t,T}\phi,e_k\rangle| \leq C_{n}\eps^{q_{n}} \|\phi\|
    \\    \!\Rightarrow\!\!\!\!
   & \sup_{t\in [T/2,T]} \sup_{\mbox{\tiny$\begin{array}{c}
k\in \cZ_n,\ell,j \in \cZ_0  \\
m,m',m''\in \{0,1\}
\end{array}$}}
|\langle \cK_{t,T}\phi,\cos (\ell z+\frac{\pi}{2}m') \cos (j z+\frac{\pi}{2}m'')
   \cos(kz+\frac{\pi}{2}m) \rangle|\leq C_{n+1}\eps^{q_{n+1}}\|\phi\|.
   \end{split}
 \end{eqnarray*}

Therefore, based on (\ref{pp-8})(\ref{pp-7}) and  (\ref{m-1})(\ref{m-2}), we complete  the proof.

\end{proof}

We are now in a position to give \textbf{a proof of Proposition \ref{h-5}}:
\begin{proof}
First,  we recall the definition of $\Omega_{\eps, \cM}$ in Lemma \ref{h-10} and let $C_0=1,q_0=\frac{1}{8}$.
For any $n\in \mN,$ after we   have defined     the constant $C_{n},q_{n},$
 we define $p_{n+1},q_{n+1},C_{n+1},\Omega_{\eps,n}$ by Lemma \ref{h-11}.
% Recursively, for any $n\geq 1,$ $\Omega_{\eps,n},C_n,  q_n$ are  well defined.

Let
\begin{eqnarray*}
\Omega_{\eps}= \Omega_{\eps, \cM}\cap \cap_{n=1}^{N} \Omega_{\eps, n}.
\end{eqnarray*}
Noting $\cK_{t,T}\phi=\phi$ for $t=T,$ by Lemmas \ref{h-10}, \ref{h-11},
for some positive constants   $p_N^*, q_N^*$,
 $C=C(T,N)$, we have
  \begin{eqnarray*}
    \mP((\Omega_\eps)^c )\leq C\eps ^{p_N^*},
  \end{eqnarray*}
  and
  \begin{eqnarray*}
\langle \cM_{0,T}\phi,\phi\rangle \leq \eps \|\phi\|^2 ~\Rightarrow~ \langle \cQ_N\phi,\phi\rangle \leq  C \eps^{q_N^*}\|\phi\|^2,
  \end{eqnarray*}
  which is valid  on the set  $\Omega_\eps$  for any  $\phi \in\cS_{\alpha,N}.$  The proof of  Proposition \ref{h-5} is finished.
\end{proof}

\subsection{Proof of Theorem \ref{p-9}}\label{p-59}
The aim of this subsection is to  give the proof of Theorem   \ref{p-9}.
\begin{proof}
Let  $\Omega_\eps$ be a set  given by  Proposition \ref{h-5}.
Let $\eps^{*}$  be a constant such that  for any $\eps\in (0,\eps^*]$
\begin{eqnarray}\label{p-56}
\frac{\alpha}{2}>C_2 \eps^{q_2},
\end{eqnarray}
where, $C_2,q_2$ are the  constants appeared  in   Proposition  \ref{h-5}.

By Proposition \ref{h-5},   for some $C_1,q_1>0,$ we have
 \begin{eqnarray*}
    \mP((\Omega_\eps)^c )\leq C_1\eps ^{q_1}.
  \end{eqnarray*}
On the set  $\Omega_\eps$,  for any  $\phi \in\cS_{\alpha,N},$ if
\begin{eqnarray*}
  \langle \cM_{0,T}\phi,\phi\rangle< \eps \|\phi\|^2,
\end{eqnarray*}
 by Proposition \ref{p-53} and Proposition \ref{h-5}, we have
\begin{eqnarray*}
\frac{\alpha}{2}\|\phi\|^2\leq \langle \cQ_N\phi,\phi\rangle \leq C_2\eps^{q_2}\|\phi\|^2,
\end{eqnarray*}
which contradicts	  with   (\ref{p-56}).  Therefore, (\ref{p-55}) holds on the set
$\Omega_\eps$ and we complete the proof of Theorem \ref{p-9}.
\end{proof}
 Using  Theorem \ref{p-9},  we obtain  the following gradient estimate. The method to prove the this   Proposition  is classical in this paper.  One can see
 \cite{FGRT}\cite{martin}\cite{Hairer02}\cite{Hairer} etc.

\begin{proposition}\label{p-28}
  For some $\gamma_0>0$ and every $\eta>0, U_0\in H$, the Markov semigroup $\{P_t\}_{t\geq 0}$ defined by (\ref{p-26}) satisfies the following estimate
  \begin{eqnarray*}
    \|\nabla P_t\Phi(U_0)\|\leq C\left(\sqrt{P_t(|\Phi|^2)(U_0)}+e^{-\gamma_0t}
    \sqrt{P_t(\|\nabla\Phi\|^2)(U_0)}\right)
  \end{eqnarray*}
  for every $t>0$ and $\Phi \in C_b(H)$, where $C$ is a constant   independent of $t,U_0$ and $\Phi.$
\end{proposition}
\begin{proof}
Our proof is very similar to that in \cite[Section 3]{FGRT} except some little changes.

We build the control $v$ and derive the associated $\rho_t=\cJ_{0,t}\xi-\cA_{0,t}v$  in (\ref{a-1}) using the same interative construction  as that in  \cite{FGRT}. Denote by $v_{s,t}$ the control $v$ restricted to the time interval $[s,t]$.
Obviously,  $\rho_0=\xi$ and $\rho_t$ depends on $\xi,t,v_{0,t}.$
  For each even non-negative integer $n\in 2\mN,$ having determined $v_{0,n}$  and $\rho_n$, we set
\begin{eqnarray*}
  v_{n,n+1}(r):&=&(\cA^{*}_{n,n+1}(\cM_{n,n+1}+I\beta)^{-1}\cJ_{n,n+1}\rho_n)(r),
\quad v_{n+1,n+2}(r)=0,
\end{eqnarray*}
for $r\in[n,n+2],$ where $\beta=\beta(n)>0$ is to be determined in (\ref{a-4}) below.

We define
\begin{eqnarray*}
  \cR_{n,n+1}^\beta:=\beta(\cM_{n,n+1}+I\beta)^{-1}.
\end{eqnarray*}
As that in \cite{FGRT}, we split $\rho_{n+2}=\rho_{n+2}^H+\rho_{n+2}^L,$ where
\begin{eqnarray}
\label{a-3}
\rho_{n+2}^H=\cJ_{n+1,n+2}Q_N \cR_{n,n+1}^\beta\cJ_{n,n+1}\rho_n,~~\rho_{n+2}^L=\cJ_{n+1,n+2}P_N \cR_{n,n+1}^\beta\cJ_{n,n+1}\rho_n.
\end{eqnarray}
By (\ref{3-1}),
for some absolute constant $C_0>1$, we have
\begin{eqnarray*}
\mE
 (1+\|U_{n+1}^8\|^2)
   \big|\cF_n\leq C_0.
\end{eqnarray*}
Set  $\delta=\frac{1}{2^9C_0}.$
By the above inequality, Lemma \ref{rr-1} and   (\ref{p-38})(\ref{6-2}), one sees that
 \begin{eqnarray}
  \nonumber  \mE  (\|\rho_{n+2}^H\|^8 |\cF_n) &\leq & \|\rho_n\|^8 \mE \Big(
   \mE \big(\|\cJ_{n+1,n+2}Q_N\|^8|\cF_{n+1}\big)\cdot \|\cJ_{n,n+1}\|^8
   \big|\cF_n\Big)
   \\ \nonumber  &\leq &
    \|\rho_n\|^8 \mE \big(
   \delta (1+\|U_{n+1}^8\|^2))\cdot \|\cJ_{n,n+1}\|^8
   \big|\cF_n\big)
   \\ \label{50-1}&\leq &C_0 \delta \|\rho_n\|^8
 \end{eqnarray}
for   appropriate $N=N(\delta).$
  Fix such an $N$ in  (\ref{a-3}).
  Following  the lines in the  \cite[Lemma 3.1]{FGRT},  noting  Lemmas \ref{p-29}-\ref{a-2} and   Theorem \ref{p-9},      there exists $\beta=\beta(n)>0$ such that
 \begin{eqnarray}
 \label{a-4}
   \mE (\|\rho_{n+2}^L\|^8 |\cF_n) &\leq & \delta \|\rho_n\|^8.
 \end{eqnarray}
By (\ref{50-1}) and  (\ref{a-4}),  we have
\begin{eqnarray*}
 \mE ( \|\rho_{n+2}\|^8 |\cF_n) \leq 2^7 \mE ( \|\rho_{n+2}^L\|^8 +\|\rho_{n+2}^H\|^8|\cF_n) \leq
 2^7 \cdot 2 C_0 \delta     \|\rho_{n}\|^8= \frac{1}{2}   \|\rho_{n}\|^8
\end{eqnarray*}
which implies that, for any  even non-negative integer $n,$ we have
\begin{eqnarray}
\label{rr-2}
\mE   \|\rho_{n}\|^8\leq 2^{-n/2}\|\xi\|^8.
\end{eqnarray}

Based on (\ref{rr-2}) the  estimates in Section 2,  following the lines in   \cite[Section 3]{FGRT}, we deduce that
 \begin{eqnarray*}
   \sup_{\|\xi\|=1,t\geq 0}\mE \left|\int_0^t v \cdot \dif W\right|\leq C
 \end{eqnarray*}
 and
 for some $\gamma_0>0,$
 \begin{eqnarray*}
   \sup_{\|\xi\|=1}\mE \|\rho_t\|^2\leq Ce^{-\gamma_0t}.
 \end{eqnarray*}
 By \cite[Section 3.1]{FGRT}, we complete our  proof.

\end{proof}

\section{Proof of Theorem  \ref{p-27}}

Let $H$ be a Banach space.
Recall that
\begin{eqnarray*}
   d(x,y)=1\wedge \delta^{-1}\|x-y\|, ~~\forall x,y\in H,
 \end{eqnarray*}
 where $\delta$ is a small parameter to be adjusted later on.
 On the set
 \begin{eqnarray*}
Pr_1(H):=\left\{\mu\in Pr(H):\int_H d(0,u)\dif \mu(u)<\infty\right\},
 \end{eqnarray*}
the    metric $d$  induces a   Wasserstein-Kantorovich  distance defined   by
\begin{eqnarray*}
  d(\mu_1,\mu_2)=\sup_{\|\Phi\|_{d}\leq 1}\left|\int_{H}\Phi(x)\mu(\dif x)-\int_{H}\Phi(x)\nu(\dif x)\right|
\end{eqnarray*}
 where $\|\Phi\|_{d}$ denotes the Lipschitz constant of $\Phi$ in the metric $d$.

We recall the following  abstract results.
Then, we give a proof of Theorem  \ref{p-27}.
\begin{thm}(See \cite[Theorem 2.5]{Hairer02}.)
\label{i-1}
Let $(P_t)_{t>0}$ be a Markov semigroup over a Banach space $H$
satisfying
\begin{itemize}
  \item[(1)]
  there   exist constants $\alpha\in  (0,1)$,  $C > 0$ and $T_1 > 0$ such
that
\begin{eqnarray}
\label{r-1}
 \|D P_t\Phi\|_\infty\leq C\|\Phi\|_\infty+\alpha_1 \|D\Phi\|_\infty,
\end{eqnarray}
for every $t \geq T_1$  and every Fr\'echet differentiable function $\Phi:H\rightarrow \mR;$
\item[(2)]
for  every $\delta > 0$, there exists a $T_2=T_2(\delta)$  so that for any
$t > T_2$ there exists an $a > 0$ so that
\begin{eqnarray}
\label{r-2}
  \sup_{\Gamma \in \cC(P_t^*\delta_{U_0},P_t^*\delta_{\widetilde{U}_0}) } \Gamma\{(U',U'')\in H\times H:\|U'-U''\|<\delta\}\geq a,
\end{eqnarray}
for every $U_0,\widetilde{U}_0\in H.$
Here $\delta_U$ is the dirac measure concentrated at $U$, the operator $P_t^*$ is defined by (\ref{e-1}) and  $\cC(\mu_1,\mu_2)$ denotes   the set of all measures $\pi$
on $H \times H$ such that $\pi(A \times H)=\mu_1(A)$ and
$\pi(H\times A)=\mu_2(A)$ for every Borel set~$ A \subset H,$
\end{itemize}
 Then, there exist constants $\delta > 0, \alpha < 1$ and $T > 0$
such that
\begin{eqnarray}
\label{e-2}
  d(P_T^*\mu_1,P_T^*\mu_2)\leq  \alpha d(\mu_1,\mu_2)
\end{eqnarray}
for every pair of probability measures $\mu_1,\mu_2$ on $H$. In particular, $(P_t)_{t>0}$ has
a unique invariant measure $\mu_*$ and its transition probabilities converge exponentially fast to $\mu_*.$
\end{thm}

\begin{thm}
  (See \cite[Theorem 2.1]{KW}.)
  \label{s-3}
  Let $(P_t)_{t\geq 0}$ be a Feller Markov semigroup on a metric space $(H,d)$ with the continuity property: $\lim_{t\rightarrow 0}P_t\Phi(U_0)=\Phi(U_0)$ for all $\Phi\in C_b(H),U_0\in H.$ Let $P_t(U_0,A)$ be the associated transition functions. Suppose that $(P_t)_{t\geq 0}$ satisfy
  \begin{itemize}
    \item[(1)]
     for some $C,\gamma>0$  and  every   $\mu_1,\mu_2\in Pr_1(H)$,   \
    \begin{eqnarray}
    \label{i-2}
      d(P_t^*\mu_1,P_t^*\mu_2)\leq Ce^{-\gamma t} d(\mu_1,\mu_2),
    \end{eqnarray}
    \item[(2)] for every $R>0$
  \begin{eqnarray}
  \label{s-2}
    \sup_{t\geq 0}\sup_{U_0\in B_R}\int_H |d(0,U)|^3P_t(U_0,\dif U)<\infty,
  \end{eqnarray}
  where $B_R:=\{U_0\in H,~d(0,U_0)<R\}.$
  \end{itemize}
  Then, there exists a unique invariant probability measure $\mu_*\in Pr_1(H)$ such that for any $\Phi\in C^1(H)$ and any $U_0\in H$
   \begin{eqnarray*}
     \lim_{T\rightarrow \infty}\frac{1}{T}\int_0^T \Phi(U(t,U_0))\dif t=\int_H \Phi(\bar{U})\dif \mu_*(\bar{U})=:m_{\Phi} \text{  in probability.}
   \end{eqnarray*}
   Moreover, the limit $\sigma^2=\lim_{T\rightarrow \infty}\frac{1}{T}\mE \left(\int_0^T (\Phi(U(t,U_0))-m_{\Phi}) \dif t\right)^2$ exists and
      \begin{eqnarray*}
        \lim_{T\rightarrow \infty}\mP\left(\frac{1}{\sqrt{T}}\int_0^T (\Phi(U(t,U_0))-m_{\Phi})\dif t<\xi\right)=\cX_\sigma(\xi),
      \end{eqnarray*}
      where $\cX_\sigma$ is the distribution function of a normal random variable with zero mean and variance $\sigma^2$.
\end{thm}

We are now in a position to give a proof of Theorem  \ref{p-27}.
\begin{proof}
%(a)
%By $It\^{o}$ formula, it yields
%  \begin{eqnarray*}
%   \|U_t\|^2-\|U_0\|^2+2\int_0^t \|U_s\|_1^2\dif s  &=& \mathfrak{B}_0t +2\int_0^t \langle U_s, G \dif W_s\rangle
%   \\ &&+2 \int_0^t \langle U_s,U_s\rangle\dif s-\int_0^t \langle U_s,U_s^3\rangle\dif s,
%  \end{eqnarray*}
%  where $\mathfrak{B}_0$ appears in (\ref{s-1}). Then it follows that
%\begin{eqnarray*}
%  \frac{1}{T}\mE \int_0^T \|U\|_{1}^2\dif t& \leq & \frac{1}{T}\mE \int_0^T \big[\|U\|^2+\|U\|_{1}^2\big]\dif t
%  \\ &\leq &  \frac{\|U_0\|^2}{T}+2 \cE_0+2\pi+\|U_0\|^2.
%\end{eqnarray*}
%Following the same way as that  in \cite[Page 2443]{FGRT}, one arrives at that  there exists an invariant measure for the semigroup $P_t.$

Recall that $U_t=U(t,U_0)$ is the solution of
(\ref{p-1}). For any $t>0,U_0\in H$ and $E\in \cB(H),$
$P_t(U_0,E),P_t$ and $P_t^*$ are defined by (\ref{zhang-1})-(\ref{e-1}) respectively.

We divide  our proof into two parts (a) and (b). In the first part (a), we  use Theorem \ref{i-1} to   prove   (\ref{p-31}). In the second part (b), we   use  Theorem \ref{s-3} to give a proof of    (\ref{p-35}) and (\ref{p-36}).

(a)  First, by Proposition \ref{p-28},  (\ref{r-1}) holds.
For any $r>0,$ we use $B_r$ to denote $\{U'\in H,\|U'\| \leq r \}.$  Following the same way as that  in \cite[Page 2489]{FGRT}, one arrives at that   for any  $\gimel,\delta>0$ there exists $T_*=T_*(\gimel,\delta)\geq 0$ such that
\begin{eqnarray}
\label{r-3}
\inf_{\|U\|\leq \gimel } P_T(U, B_\delta)>0,
\end{eqnarray}
for any $T>T_*.$

By Lemma \ref{4-1}, there exist positive  constants  $C_1$ and $\gamma$ such that  for any $\gimel,\delta>0$ and  $T>t=1$, we have
\begin{eqnarray}
  \nonumber && P_T\Big(\overline{U}_0,B_{\frac{\delta}{2}}\Big)=\int_{H}P_t\left(\overline{U}_0,\dif U\right)P_{T-t}\Big(U, B_{\frac{\delta}{2}}  \Big)
  \\ \nonumber  &&\geq \int_{B_\gimel}P_t\left(\overline{U}_0,\dif U\right)\inf_{U\in B_\gimel }P_{T-t}\Big(U,B_{\frac{\delta}{2}}  \Big)
  \geq (1-\frac{\mE\|U_t\|}{\gimel})\inf_{U\in B_\gimel }P_{T-t}\Big(U, B_{\frac{\delta}{2}} \Big)
    \\ \label{y-1} &&\geq (1-\frac{C_1(1+t^{-\gamma})}{\gimel})\inf_{U\in B_\gimel }P_{T-t}\Big(U, B_{\frac{\delta}{2}}  \Big)
       \geq (1-\frac{2C_1 }{\gimel})\inf_{U\in B_\gimel }P_{T-t}\Big(U, B_{\frac{\delta}{2}} \Big),
\end{eqnarray}
where $U_t$ is the solution to equation (\ref{p-1}) with initial value $\overline{U}_0.$
In the above inequality, we set $\gimel=4C_1$. By (\ref{r-3}), there exists  $T^*=T^*(\gimel,\delta)$ such that for any $T>T^*,$
\begin{eqnarray*}
\inf_{\|U \|\leq \gimel } P_{T-t}(U, B_{\frac{\delta}{2}})>0.
\end{eqnarray*}
 Combining  the above inequality with (\ref{y-1}),  noting  $\gimel=4C_1$,  one arrives at that
\begin{eqnarray}
\label{u-1}
\inf_{ \overline{U}_0 \in H  } P_{T}(\overline{U}_0, B_{\frac{\delta}{2}})>0
\end{eqnarray}
for $T\geq T^*.$

For any $U_0,\widetilde{U}_0\in H$ and $T>0$, we define  $\tilde{\Gamma}_{U_0,\widetilde{U}_0}\in Pr(H\times H)$ by
\begin{eqnarray*}
\tilde{\Gamma}_{U_0,\widetilde{U}_0}(A_1\times A_2):=P_T(U_0,A_1)P_T(\widetilde{U}_0,A_2)\quad \text{ for any } A_1,A_2\in \cB(H).
\end{eqnarray*}
Then, by (\ref{u-1}),  we have
\begin{eqnarray*}
  && \sup_{\Gamma \in \cC(P_t^*\delta_{U_0},P_t^*\delta_{\widetilde{U}_0})} \Gamma\{(U',U'')\in H\times H:\|U'-U''\|<\delta\}
  \\ && \geq \tilde{\Gamma}_{U_0,\widetilde{U}_0}\left\{ (U',U'')\in H\times H:\|U'-U''\|<\delta \right\}
  \\ &&\geq P_T\left(U_0,B_{\frac{\delta}{2}}\right)\cdot P_T\left(\widetilde{U}_0,B_{\frac{\delta}{2}}\right)
  \geq \left(\inf_{ \overline{U}_0 \in H  } P_{T}(\overline{U}_0, B_{\frac{\delta}{2}})\right)^2>0
\end{eqnarray*}
which yields (\ref{r-2}).

By  Theorem \ref{i-1}, for some $\alpha<1,T>0$ and every pair of probability measures
$\mu_1,\mu_2$ on $H$, we have
\begin{eqnarray*}
  d(P_T^*\mu_1, P_T^*\mu_2)\leq \alpha d(\mu_1,\mu_2).
\end{eqnarray*}
Therefore, for some $C,\gamma>0$  and  every   $\mu_1,\mu_2\in Pr_1(H)$,   we have
    \begin{eqnarray}
    \label{iii-2}
      d(P_t^*\mu_1,P_t^*\mu_2)\leq Ce^{-\gamma t} d(\mu_1,\mu_2).
    \end{eqnarray}
Also by Theorem \ref{i-1},  $(P_t)_{t>0}$ has
a unique invariant measure $\mu_*.$

In (\ref{iii-2}), letting $\mu_1=P_t^*\delta_{U_0}$ and $\mu_2=\mu_*$, one sees that
  \begin{eqnarray*}
      d(P_t^*\delta_{U_0},P_t^*\mu^*)\leq Ce^{-\gamma t} d(\delta_{U_0},\mu^*),
    \end{eqnarray*}
    which implies
    \begin{eqnarray*}
  \sup_{\|\Phi\|_{d}\leq 1}\left|P_t\Phi(U_0)-\int_{H}\Phi(z)\mu^*(\dif z)\right|\leq  Ce^{-\gamma t}.
\end{eqnarray*}
We complete the proof of  (\ref{p-31}).

(b) By $It\^{o}$ formula and (\ref{p-1}), for any $\eta>0,$ it gives that
  \begin{eqnarray*}
  &&  \nonumber  \eta \|U_t\|^2-\eta \|U_0\|^2+2\eta \int_0^t \| U_s\|_1^2\dif s
    \\ \nonumber  && = \eta \cE_0t +2\eta \int_0^t \langle U_s, G\dif W_s\rangle
    +2\eta \int_0^t \langle U_s,U_s\rangle\dif s-2\eta\int_0^t \langle U_s,U_s^3\rangle\dif s
   \\   && \leq  \eta (\cE_0+4 \pi)t +2\eta \int_0^t \langle U_s, G\dif W_s\rangle
   -2\eta \int_0^t \|U_s(z)\|^2\dif s.
  \end{eqnarray*}
  Let  $\bar{U}(t)=\eta \|U_t\|^2, \bar{Z}(t)=\eta  \|U_t\|_1^2+\eta \|U_t\|^2 $,
  then we have
  \begin{eqnarray*}
    \eta (\cE_0+4 \pi)-2\eta \|U_s\|_1^2-2\eta \|U_s\|^2
    &\leq  &   \eta (\cE_0+4 \pi) -2\bar{Z}(t),
    \\ 4\eta^2 |\langle U_s,G\rangle|^2 &\leq &  4\eta \cE_0  \bar{Z}(t).
  \end{eqnarray*}
By \cite[lemma 5.1]{Hairer02},   there exists $\eta^*>0,$ such that for any $\eta\in (0,\eta^*]$
  \begin{eqnarray*}
    \mE\[\exp\Big\{\eta \|U_t\|^2_H+\frac{1}{2}e^{-t/2}\int_0^t \eta \| U_s\|^2_1\dif s \Big\}\]\leq C(\eta,\cE_0)\exp\{\eta \|U(0)\|^2e^{- t}\},
  \end{eqnarray*}
  which yields (\ref{s-2}).

The  Feller property and stochastic continuity of $P_t$ follow immediately from the well-posedness properties of (\ref{p-1}) as recalled in Proposition \ref{9-1}.

Therefore, by (\ref{iii-2}) and the arguments above,    the conditions of  Theorem \ref{s-3} hold  for $P_t$ and   we finish the proof of   (\ref{p-35}) and (\ref{p-36}).

\end{proof}

\appendix

%===========================================================================


\begin{thebibliography}{WWW}
\bibitem{ADX}
Albeverio, S., Debussche, A.,  Xu, L. (2012). Exponential mixing of the 3d stochastic navier-stokes equations driven by mildly degenerate noises. \emph{Applied Mathematics} \&\emph{ Optimization}, 66(2), 273-308.


% \bibitem{liuwei}
%Brze\'{z}niak, Z.,  Liu, W.,  Zhu, J.(2011).    Strong solutions for SPDE with locally monotone coefficients driven by L\'{e}vy  noise[J].  arXiv preprint arXiv:1108.0343.



\bibitem{BHZ}
Brze\'zniak, Z. , Hausenblas, E. ,  Zhu, J. . (2013). 2d stochastic navier-stokes equations driven by jump noise.  \emph{Nonlinear Analysis: Theory, Methods} \& \emph{Applications}, 79, 122-139.


%\bibitem{BDP}
%Barbu, V., \& Prato, G.D.  (2007). Existence and ergodicity for the two-dimensional stochastic magneto-hydrodynamics equations. \emph{Applied Mathematics and Optimization}, 56(2), 145-168.






\bibitem{P} Constantin, P.,  Foias, C.,  (1988). Navies-Stokes equations. Chicago Lectures In Mathematics. University of Chicago Press, Chicago.




\bibitem{DX}
Dong, Z., Xie, Y. (2011).  Ergodicity of stochastic 2D Navier-Stokes equation with L\'evy noise. \emph{J. Differential Equations},  251(1), 196-222.

\bibitem{E-2000}
E, W. (2000). Stochastic hydrodynamics. \emph{Current Developments in Mathematics}, 2000(1), 109-147.


\bibitem{DXZ}
Dong, Z. , Xu, L.,  Zhang, X. . (2011). Invariant measures of stochastic 2d navier-stokes equation driven by $\alpha$-stable processes. \emph{Electronic communications in probability}, 16.


\bibitem{weinan}
E, W., Mattingly, J. C.  (2010). Ergodicity for the navier-stokes equation with degenerate random forcing: finite-dimensional approximation. \emph{Communications on Pure} \& \emph{Applied Mathematics}, 54(11), 1386-1402.





%\bibitem{mattingly}
% E, W., Mattingly, J. C. , \& Sinai, Y.  (2001). Gibbsian dynamics and ergodicity?for the stochastically forced navier¨Cstokes equation. \emph{Communications in Mathematical Physics}, 224(1), 83-106.


\bibitem{GM}
Goldys, B.,  Maslowski, B.(2005).
Exponential ergodicity for stochastic Burgers and 2D Navier-Stokes equations.
\emph{J. Funct. Anal.} 226(1), 230-255.

 \bibitem{FGRT}
 F\"oldes, J., Glatt-Holtz,  N., Richards,  G., Thomann, E. (2005).  Ergodic and mixing properties of the Boussinesq equations with a degenerate random forcing[J]. \emph{Journal of Functional Analysis},   269(8):2427-2504.

%\bibitem{Ga-Ma}
%Gawarecki, L. and Mandrekar V.: Stochastic Differential Equations in Infinite Dimensions. Springer. 2010.

\bibitem{Hairer-2001}
Hairer, M. (2001).  Exponential mixing properties of stochastic pdes through asymptotic coupling. \emph{Probability Theory} \& \emph{ Related Fields}  124 (3), 345-380.


\bibitem{martin}
Hairer, M., Mattingly,  J.C.(2006).
Ergodicity of the 2D Navier-Stokes Equations with Degenerate
Stochastic Forcing, \emph{Annals of Mathematics} , 164(3), 993-1032.

\bibitem{Hairer02}Hairer, M., Mattingly,  J. C. (2008). Spectral gaps in Wasserstein distances and the 2D stochastic Navier-Stokes equations.
\emph{The Annals of Probability},   36(6), 2050-2091.


%\bibitem{HS}
%Huang, J. and Shen, T.(2016) Well-posedness and dynamics of the stochastic fractional magneto-hydrodynamic equations ¡î[J]. \emph{Nonlinear Analysis}, 133:102-133.

\bibitem{Hairer}Hairer, M., Mattingly,  J. C. (2011).  A theory of Hypoellipticity and Unique Ergodicity for
Semilinear Stochastic PDEs, \emph{Electronic Journal of Probability}, 16, 658-738,

\bibitem{KW}
Komorowski,  T., Walczuk,  A.  (2012). Central limit theorem for Markov processes with spectral gap in the Wasserstein metric[J]. \emph{Stochastic Processes} \& \emph{Their Applications},  122(5):2155-2184.



%\bibitem{K-S}
%Kuksin, S. ,
%\& Shirikyan, A. .(2002). Coupling approach to white-forced nonlinear pdes. \emph{JOURNAL DE MATHEMATIQUES PURES ET APPLIQUEES}.





\bibitem{KS}
Kuksin, S., Shirikyan, A. (2001).
A Coupling Approach
to Randomly Forced Nonlinear PDEs.I, \emph{Commun. Math. Phys.} 221, 351-366.

%\bibitem{KS2002}
%Kuksin, S., Shirikyan, A. (2002).  Coupling approach to white-forced nonlinear PDEs. \emph{J. Math. Pures Appl.},    81(6), 567-602.


\bibitem{KS2012}
Kuksin,  S., Shirikyan,  A. (2012). Mathematics of Two-Dimensional Turbulence[J]. Cambridge Tracts in Mathematics.


\bibitem{KPS}
 Komorowski, T., Peszat, S., Szarek, T. (2010).   On ergodicity of some Markov processes. \emph{Ann. Probab.},  38(4), 1401-1443.

%\bibitem{LZ}
%Lei, Z. and Zhou, Y. (2009). BKM's criterion and Global weak solutions for magnetohydrodynamics with zero viscosity, \emph{Discrete Contin. Dyn.Syst.},   25,   575-583.
\bibitem{LR}
Liu, W.,  R\"ockner, M.(2013). Local and global well-posedness of SPDE with generalized coercivity conditions. \emph{J. Differential Equations}, 254(2), 725-755.


\bibitem{Menaldi}
Menaldi, J.-L., Sritharan,  S.S.(2002). Stochastic 2D Navier-Stokes equation, \emph{Appl. Math.
Optim.},  46, 31-53.




\bibitem{Matt2002}
Mattingly, Jonathan C. (2002).
Exponential convergence for the stochastically forced Navier-Stokes equations and other partially dissipative dynamics.  \emph{Comm. Math. Phys.} 230, no. 3, 421-462.



%\bibitem{mat02}  Mattingly, J.C.(2002).
%The dissipative scale of the stochastics Navier-Stokes equation:
%regularization and analyticity,\emph{ J. Statist. Phys.}  108 ,
%1157-1179.



\bibitem{MZ}
Mohammed, Salah., Zhang, T.  (2013).  Anticipating stochastic 2D Navier-Stokes equations. \emph{J. Funct. Anal.}  264(6), 1380-1408.

%\bibitem{MZ2}
%Mohammed, S., Zhang, T.(2010). Dynamics of stochastic 2D Navier-Stokes equations. \emph{J. Funct. Anal.}  258, 3543-3591.

\bibitem{MY-2002}
Masmoudi, N.,  Young, L.S. (2002).  Ergodic theory of infinite dimensional systems with applications to dissipative parabolic pde¡¯s. \emph{Communications in Mathematical Physics}  227 (3), 461-481.

\bibitem{MW17}
Mourrat, J.C.,  Weber, H. (2017). The dynamic $\Phi^4_3$  model comes down from infinity.\emph{ Communications in Mathematical Physics} (3).







\bibitem{Nualart}Nualart, D. (2005). The Malliavin Calculus and Related Topics, second edition, Springer.





%\bibitem{O07}
%  Odasso, C. (2007). Exponential mixing for 3D stochastic Navier-Stokes equations, \emph{  Communications in mathematical physics} 270, 109-139,


\bibitem{O2}
Odasso, C. (2008). Exponential mixing for stochastic pdes: the non-additive case. \emph{Probability Theory and Related Fields}, 140(1-2), 41-82.



\bibitem{O3}
 Odasso, C. (2006).  Ergodicity for the stochastic Complex Ginzburg Landau equations[C],
\emph{Annales de l'Institut Henri Poincare (B) Probability and Statistics}. No longer published by Elsevier, 42(4),  417-454.


\bibitem{PZ-92}
Da Prato, G.,  Zabczyk, J. (1992). Stochastic equations in infinite dimensions. Cambridge Univ. Pr.

\bibitem{PZ-96}
Da Prato, G.,
 Zabczyk, J. (1996). Ergodicity for Infinite Dimensional Systems. Ergodicity for infinite dimensional systems /. Cambridge University Press.


\bibitem{RZ}
R\"ockner, M.,  Zhang, X. (2009).  Stochastic tamed 3D Navier-Stokes equations: existence, uniqueness and ergodicity. \emph{Probab. Theory Related Fields} 145, 211-267.


\bibitem{AS}
Shirikyan, A. (2008). Exponential mixing for randomly forced partial differential equations: method of coupling. \emph{Instability in Models Connected with Fluid Flows} II, 7, 155-188.

%
% \bibitem{WJ}
% Wu,J.:Generalized MHD equations, J.Differential Equations, 195(2003),284-312.

% \bibitem{ST}
% Sermange, M. and Temam, R.: Some mathematical questions related to the MHD equations. \emph{ Comm.Pure Appl.Math.}36(1983)635-664.

 \bibitem{SH2017}
 Shen,  T., Huang,  J. (2017). Ergodicity of stochastic Magneto-Hydrodynamic equations driven by $\alpha$-stable noise[J]. \emph{Journal of Mathematical Analysis} \& \emph{Applications}, 446(1):746-769.

% \bibitem{SHZ}
% Shen, T., Huang, J. and Zeng, C.: Ergodicity of the 2D stochastic fractional Magneto-hydrodynamic equations driven by degenerate multiplicative noise. Preprint.

\bibitem{XZ}
 Xu, L.,  Zegarlinski, B.(2010).  Existence and exponential mixing of infinite white stable systems with unbounded interactions. \emph{Electron. J. Probab.} 15, 1994-2018.

\bibitem{Xu-2012}
Xu, L. (2012). Ergodicity of the stochastic real ginzburg-landau equation driven by $\alpha$-stable noises. \emph{Stochastic Processes} \&  \emph{Their Applications}, 123(10), 3710-3736.

\end{thebibliography}
\end{document}